\documentclass{amsart}
\usepackage{amssymb,amstext,amsthm,titletoc,fancyhdr,color,ulem}
\newtheorem{theorem}{Theorem}[section]

\theoremstyle{definition}
\newtheorem{definition}[theorem]{Definition}
\newtheorem{example}[theorem]{Example}

\theoremstyle{remark}
\newtheorem{remark}[theorem]{Remark}

\numberwithin{equation}{section}

\def\C{{\mathbb C}}

\def\Z{{\mathbb Z}}
\def\Im{\mathop{\rm Im}}

\newcommand{\beqnn}{\begin{equation}}
\newcommand{\eeqnn}{\end{equation}}
\newcommand{\eb}{\begin{enumerate}}
\newcommand{\ee}{\end{enumerate}}
\newcommand{\bbm}{\begin{bmatrix}}
\newcommand{\ebm}{\end{bmatrix}}
\newcommand{\bpm}{\begin{pmatrix}}
\newcommand{\epm}{\end{pmatrix}}

\newcommand{\ra}{\rightarrow}

\newcommand{\ts}{\textstyle}
\newcommand{\bi}{\begin{itemize}}
\newcommand{\ei}{\end{itemize}}
\newcommand{\beq}{\begin{eqnarray*}}
\newcommand{\eeq}{\end{eqnarray*}}
\def\sn{\mathop{\rm sn}} \def\sc{\mathop{\rm sc}} \def\sd{\mathop{\rm sd}}
\def\cn{\mathop{\rm cn}} \def\cs{\mathop{\rm cs}} \def\cd{\mathop{\rm cd}}
\def\dn{\mathop{\rm dn}} \def\dc{\mathop{\rm dc}} \def\ds{\mathop{\rm ds}}
\def\nd{\mathop{\rm nd}} \def\nc{\mathop{\rm nc}} \def\ns{\mathop{\rm ns}}

\definecolor{darkgreen}{rgb}{0,0.6,0.1}

\newcommand{\beqq}{\begin{eqnarray}}
\newcommand{\eeqq}{\end{eqnarray}}
\newcommand{\beqn}{\begin{eqnarray}}
\newcommand{\eeqn}{\end{eqnarray}}

\begin{document}

\title{Symmetries of the Darboux Equation}

\author{Yik-Man Chiang}
\address{Department of Mathematics, Hong Kong University of Science and Technology,
Clear Water Bay, Kowloon, Hong Kong}
\email{machiang@ust.hk}
\thanks{The first and third authors are partially supported by Hong Kong Research Grant Council project no. 601111 and 16300814}

\author{Avery Ching}
\address{Department of Mathematics, Hong Kong University of Science and Technology, Clear Water Bay, Kowloon, Hong Kong}
\email{maaching@ust.hk}

\author{Chiu-Yin Tsang}
\address{Department of Mathematics, Hong Kong University of Science and Technology, Clear Water Bay, Kowloon, Hong Kong}
\email{macytsang@ust.hk}



\date{May 16, 2017}



\begin{abstract}
This paper establishes the symmetries of Darboux's equations (1882) on tori.
We extend Ince's work (1940) by developing new infinite series expansions in terms of Jacobi elliptic
functions around each of the four regular singular points of the
Darboux equation which are located at the four half-periods of the torus. The
symmetry group of the Darboux equation is given by the Coxeter group
$B_4\cong G_{\mathrm{I}}\rtimes_{\Gamma} G_{\mathrm{II}}$ where the
actions of $G_{\mathrm{I}}$ correspond to the sign changes of the parameters in the
solutions of the equations, which have no effective changes on the
equation itself, while the actions of $G_{\mathrm{II}}$ permute the four
half-periods of the torus, which are described by a short-exact sequence.
We are able to clarify the symmetries of the equation when ordered bases of the underlying torus change.
Our results show that it is much more transparent to
consider the symmetries of the Darboux equation on a torus than on the Riemann sphere $\mathbb{CP}^1$
as was usually considered by earlier researchers such as
Maier (2007). We list the symmetry tables of the Darboux equation in
both the Jacobian form and Weierstrass form. A consolidated list of 192
solutions of the Darboux equation is also given. We then consider the
symmetries of several classical equations (e.g. Lam\'e equation) all of which are
special cases of the Darboux equation. Those terminating solutions of the Darboux equation generalize the classical Lam\'e polynomials.
\end{abstract}

\maketitle



\tableofcontents

\section{Introduction}
\subsection{The Darboux equation 
}

The symmetries of the classical Lam\'e equation \cite{Lame} (1837)
		\beqn 
			\frac{{d}^{2}y}{{du}^{2}}+\big(h-\nu(\nu+1)k^2{\ts\sn^2u}\big)y=0,
		\eeqn
was first discussed by Arscott and Reid \cite[(1971)]{AR}\footnote{Although the discussion of geometric reasoning was absent, namely the ordered bases of the underlying torus. See \S3.}
 despite the long history of the equation.
In this case the symmetry group is  isomorphic to $S_3$ {(see \S\ref{LameEqn} for more detail)}.

We characterize the symmetries of the Darboux equation \cite{Darboux} (1882) \footnote{ Also known as {D}arboux-{T}reibich-{V}erdier equation \cite{Veselov}.} 
\beqn\label{E:darboux}
&&\frac{{d}^{2}y}{{du}^{2}}+\Big[h-{\xi(\xi+1)}\,{\ts\ns^2(u,k)}-{\ts\eta(\eta+1)\dc^2(u,k)}\nonumber\\&&\ \hspace{2cm} -
{\ts\mu(\mu+1)k^2\cd^2(u,k)}-\nu(\nu+1)k^2{\ts\sn^2(u,k)}\Big]\,y=0,
\eeqn
which includes the Lam\'e equation as a special case, in this paper to be isomorphic to the \textit{hyperoctahedral (Coxeter) group} $B_4$, where $\ns(u,\,k),\, \dc(u,\, k)$, $\cd(u,\, k),\,\sn(u,\, k)$ are Jacobi elliptic functions with elliptic modulus $k$
($k^2\neq 0,\pm1)$.
Notice that $\ns^2(u,k)$, $\dc^2(u,k)$, $\cd^2(u,k)$, $\sn^2(u,k)$ are doubly periodic functions with periods
$\omega_1=2K(k),\, \omega_2=2iK'(k)$ and have a double pole at $u=0,\, K(k),\, K(k)+iK'(k)$ and $iK'(k)$ respectively on the fundamental parallelogram
$P(\omega_1,\omega_2)=\{s\omega_1+t\omega_2:s,\, t\in[0,\, 1)\}$. The Darboux equation (\ref{E:darboux}) is defined on a torus of $\mathbb{C}$ modulo the lattice $\Lambda=\{m\omega_1+n\omega_2:m,n\in\mathbb{Z}\}$ , which
can be specified by the Riemann $P$-scheme
	\[
		P_{\mathbb{C}\slash\Lambda}
			\begin{Bmatrix}
\ 0\ &\ K(k)\ &\ K(k)+iK'(k)\ &\ iK'(k)\ &\\
\xi+1&\eta+1&\mu+1&\nu+1&u;\, h\\
-\xi&-\eta&-\mu&-\nu&
			\end{Bmatrix},
	\]
where the entries on the top row represent the locations of the regular singularities and the entries of the succeeding two rows under the corresponding singularities represent the two exponents of the local solutions there, and $u$, $h$ are the independent variable and the accessory parameter respectively.  We have obtained local series solutions about the singularities on a torus by developing an infinite series expansion theory in terms of
Jacobi elliptic functions. For example, the series expansion about the singular point $u=0$ with the exponent $\xi+1$ is given by
	\[
		\sn(u,k)^{\xi+1}\cn(u,k)^{\eta+1}\dn(u,k)^{\mu+1}\sum_{m=0}^\infty C_m\sn(u,k)^{2m},
	\]
where the coefficients $C_m$ satisfy a certain three-term recursion relation and the series converges in certain region on the torus (see \S\ref{S:expansion}). The idea of developing such a series expansion solution goes back to Ince \cite[(1940)]{Ince1} in a study of the well-known Lam\'e equation which has only one regular singular point on a torus.  We have adopted this idea for the Darboux differential equation which has four elliptic coefficients and each of which contributes a regular singular point within its fundamental region. We generate a total of 192 local series solutions in terms of Jacobi elliptic functions by the symmetry group of the Darboux equation to be explored in this paper.

Maier {(\cite[Theorem 4.3]{Maier1}, \cite[Theorem 3.1]{Maier})} has recently clarified that the automorphism group of the second-order Fuchsian type differential equation (FDE) {in asymmetric form} with $n\geq 3$ regular singular points {$0,\ 1,\ \infty,\ a_1,\cdots,\ a_{n-3}$} in $\mathbb{CP}^1$
is the Coxeter group $B_n\cong (\mathbb{Z}_2)^n\rtimes_{\Gamma} ~S_n$
(or {$D_{n}\cong (\mathbb{Z}_2)^{n-1}\rtimes_{\Gamma} ~S_n$} if the transformation $[\infty]_-$  interchanging  the exponents $\alpha,\beta$ at $\infty$ is excluded). {While for FDE in symmetric form, the action of the part $({\mathbb Z}_2)^n$ on the equation is trivial.}
When $n=3$, a FDE in asymmetric form is the Gaussian hypergeometric equation, which has three regular singular points $\{0,\, 1,\, \infty\}$ on $\mathbb{CP}^1$. Its symmetry group is well-known to be isomorphic to {$S_4\cong (\mathbb{Z}_2)^2\rtimes_{\Gamma} ~S_3$} (if $[\infty]_-$ is excluded), and its 24 local solutions were worked out by Kummer (see for example, Whittaker and Watson \cite[pp. 284-285]{WW}). In the case of the Heun equation  \cite[(1889)]{Heun},  which corresponds to the FDE in asymmetric form with four regular singular points $\{0,\,1,\, a,\, \infty\}$, has its symmetry group isomorphic to $(\Z_2)^4\rtimes_{\Gamma} S_4$ (if $[\infty]_-$ is included). Maier also provided a complete list of the 192 local solutions and the corresponding accessory parameters. We exhibit in the case of $n=4$ that it is more natural, as far as the symmetry is concerned, to consider Fuchsian differential equations on a torus $\mathbb{C}\slash\Lambda$ and to place the four regular singular points at the corresponding four half-periods.

The Heun equation
	\begin{equation}
		\label{heun}
			\frac{d^2y}{dt^2}+\Big(\frac{\gamma}{t}+\frac{\delta}{t-1}+\frac{\epsilon}{t-a}\Big)\frac{dy}{dt}+
\frac{\alpha\beta t-q}{t(t-1)(t-a)}y=0,
	\end{equation}
where the parameters satisfy the Fuchsian constraint $\alpha+\beta-\gamma-\delta-\epsilon+1=0$
that Maier considered has the Riemann $P$-scheme
	\[
		P_{\mathbb{CP}^1}
		\begin{Bmatrix}
\ 0\ &\ 1\ &\ a\ &\ \infty\ &\\
0&0&0&\alpha&t;\, q\\
1-\gamma&1-\delta&1-\epsilon&\beta&
	\end{Bmatrix}.
	\]

Although the Heun equation is connected to the Darboux equation by a simple change of dependent and independent variables,
we can replace the symmetry group  $(\Z_2)^4\rtimes_{\Gamma} S_4$ for the Heun equation, found by Maier, by $G_{\mathrm{I}}\rtimes_{\Gamma} G_{\mathrm{II}}$ for the Darboux equation, where the group $G_\mathrm{I}\cong (\Z_2)^4$ whose action is analogues to that of the Heun equation, while
the group $G_{\mathrm{II}}$ satisfies the short-exact sequence
	$$
		0\to K\to G_{\mathrm{II}}\to\mbox{anh}\to 0.
	$$
Here the group $K$ denotes the \textit{Klein four-group} which plays the role of translations of the half-periods (and so on the regular singular points) on the torus, while the $\mathrm{anh}$ denotes the \textit{anharmonic group} which is isomorphic to $\mathrm{Aut(X(2))}$, where $\mathrm{X}(2)=\mathbb{H}/\Gamma(2)$ and $\mathbb{H}$ is the upper half-plane of $\mathbb{C}$ (see \cite[p. 9]{Lehner1964}). The group $\mathrm{Aut(X(2))}$ is
parametrized by the quotient of the \textit{full-modular group} $\Gamma=\textrm{SL}(2,\, \mathbb{C})$ by $\Gamma(2)$, the principal congruence subgroup of level 2, which plays the role of permutations of the half-periods of the torus. The group $G_{\mathrm{II}}$ thus describes the combined actions of the half-periods of $K$ and  $\mathrm{anh}$. This paradigm allows us to see the symmetry group $G_\mathrm{II}$ acting on the half-period of the torus in a natural way. Hence it is most natural to consider Fuchsian equations with four regular singular points on a torus.  Although it is intrinsic in Maier's work  that one can consider the six equivalent classes of arrangements of the four regular singular points in $\mathbb{CP}^1$, with each of these classes characterized by the cross-ratio of the four singular points, it is much more subtle and less natural compared to putting the four points at the four half-period of a torus.
We also exhibit that Maier's hyper-Kummer group $\mathfrak{K}$ is isomorphic to our group~$G_\mathrm{II}$.

Moreover, the stumbling block (see Maier \cite[p. 812]{Maier1}) of describing the accessory parameters (eigenvalues) corresponding to the 192 local solutions of the Heun equation is reduced to only six classes (Table \ref{T2}) when considered in the Jacobian form (\ref{E:darboux}) of the Darboux equation, and to a mere of three classes (Table \ref{TT1}) when considered in the Weierstrass form (\ref{E:darbouxW}). This is because the Heun equation is written in  \textit{asymmetric form} (as in the case of the hypergeometric equation), while the Darboux equation \eqref{E:darboux} that is commonly refereed to in the literature is written in the \textit{symmetric form}. The part $(\mathbb{Z}_2)^{n-1}$ of its symmetry group  $D_{n}\cong (\mathbb{Z}_2)^{n-1}\rtimes_{\Gamma} ~S_n$ of a Fuchsian equation on $\mathbb{CP}^1$ in the symmetric form does not affect its accessory parameter. Although the Darboux equation in both the Jacobian form and Weierstrass form are in symmetric forms, these reductions are still substantially more intuitive compared to the symmetric form of the Heun equation as described in \cite[Eqn(2.2)]{Maier1}. We shall consider both the Jacobian and Weierstrass forms of the Darboux equation (in the symmetric form) in this paper. We have listed the corresponding asymmetric form in Jacobian form of the Darboux equation in Appendix \ref{S:Sparre}. Another equation named after Sparre (1883) is related to the Darboux equation in Jacobian form is parallel to the Papperitz form of the hypergeometric equation (see Whittaker and Watson \cite[p. 206]{WW} or Poole \cite[p. 86]{Poole}) will also be discussed there. Both the symmetric and the asymmetric forms of the Darboux equation are special cases of the Sparre equation.

The subgroup $G_\mathrm{II}$ acts on Weierstrass form in a simple manner while its action is more subtle on the Jacobian form. However, the study of the solutions of the Darboux equation in Jacobian form is preferred. It is because formulas of Jacobi elliptic functions are better developed (see e.g. \cite[XXII]{WW}, \cite{BF}), which make practical computation easier. A subtle point in considering the symmetries of the Darboux equation arises is that the group actions from the $G_\mathrm{II}$ would permute the half-periods the torus $\mathbb{C}\slash \Lambda$
and hence causing changes of the \textit{apparent} underlying space $\mathbb{C}\slash \Lambda$ on which the equation is defined and simultaneously while they act on the equation itself. Therefore, we address how we could make this ambiguity precise by constructing a proper underlying space on which the equations live in \S\ref{tran1} when discussing the symmetries of the Darboux equation in general. Such ambiguity does not arise when considering the symmetries of the Heun equation over $\mathbb{CP}^1$.

The Darboux equation had been forgotten by the general mathematics community until it was rediscovered for over a century later, by Treibich and Verdier \cite{TV,Verdier} (without knowing Darboux's work) in the context of the finite-gap theory and algebraic geometry. Indeed, there is a simple conversion between the Jacobi elliptic functions and the Weierstrass elliptic function $\wp(u,\tau)$. The Darboux equation becomes
\beqn\label{E:darbouxW}
&&\frac{{d}^{2}y}{{du}^{2}}+\Big[h-{\xi(\xi+1)}{\ts\wp(u;\tau)}-{\ts\eta(\eta+1)\wp(u+\omega_1;\tau)}\nonumber\\&&\ \hspace{2cm} -
{\ts\mu(\mu+1)\wp(u+\omega_2;\tau)}-\nu(\nu+1){\wp(u+\omega_3;\tau)}\Big]y=0,
\eeqn
where $\omega_1,\, \omega_2,\, \omega_3$ are the half-periods of the corresponding torus.
It is in this context that the Darboux equation was rediscovered by Treibich and Verdier \cite{TV,Verdier} that the elliptic potential in (\ref{E:darbouxW}) is a finite-gap potential if and only if all four parameters $\xi,\, \eta,\, \mu,\, \nu$ are integers. We refer to \cite[\S 1]{GW2} for a brief history.

Let us consider a family of elliptic curves over the upper half-plane $\mathbb{H}$ whose
fibre at $\tau$ (together with the point at infinity) is the zero-locus of
$$
F(x,\, y)=y^2-\big(4x^3-g_2(\tau)x-g_3(\tau)\big).
$$

Explicitly, we have a family $\mathbb{E}$ (together with the point at infinity)
$$
\mathbb{E}=\big\{(x,y;\tau)\in\C^2\times \mathrm{X}(2):y^2=4x^3-g_2(\tau)\,x-g_3(\tau)\big\}
$$
on $\mathbb{CP}^2$.
The Weierstrass elliptic function $\wp$ can now be regarded as a function defined on $\mathbb{E}$, which is just the projection to the first coordinate. On the one hand, this illustrates that both the independent variable and (half-)periods of the (\ref{E:darbouxW}) are changed when acted upon by its symmetry group. This aspect is less apparent from the Jacobian form (\ref{E:darboux}) compared with the Weierstrass form (\ref{E:darbouxW}) of the Darboux equation (as the half-periods for the Jacobian elliptic functions are signified by the elliptic moduli $k,\, k^\prime$). It turns out that the Jacobian form of the Darboux equation is preferred when we are dealing with local series expansion because more formulae for computation are available as mentioned earlier and their analogy with trigonometric functions,
 while the Weierstrass form is more advantageous when dealing with symmetry properties. As we shall see later that our theory of the series expansions, written as the sums in terms of Jacobi elliptic functions has an old origin, which can be traced back to the work of Ince \cite[(1940)]{Ince1} for solving a special case of the Darboux equation, namely the Lam\'e equation. As a result, we will present our main results for both the Jacobian form and Weierstrass form equations, but we  present the argument leading to these results for the Jacobian form (\ref{E:darboux}) only in this paper.

Indeed, amide the pioneering work of Treibich and Verdier, the equation has been under intense study from different perspectives by recent researchers
such as Gesztesy and Weikard \cite[(1996)]{GW2}, Matveev and Smirnov \cite[(2006)]{MS} , Takemura \cite[(2003, 2008, 2009)]{Take,Take2,Take3} ,
Veselov \cite[(2011)]{Veselov}  etc. On the other hand, the Lam\'e equation, which is a special case of the Darboux equation,
appears in many different contexts such as boundary value problems (Courant-Hilbert) \cite{CH}, finite-gap problem (integrable systems) \cite{Dub,Novikov,IM,GW1}, algebraic geometry \cite{BD,Katz,CLW}, Cologero-Moser-Sutherland models \cite{Take}, etc.

This paper is organised as follows. We split the discussion of the group action $G_\mathrm{I}$ of transposition of local solutions at each singularity in \S\ref{S:g1} and and $G_\mathrm{II}$ for the group action of permutations of half-periods of a torus \S\ref{tran1} and \S\ref{tran2}. More specifically, the group actions $G_\mathrm{II}$ on the Weierstrass form are given in \S\ref{tran1} and those on the Jacobian form are given in \S\ref{tran2}. The joint actions of $G_\mathrm{I}$ and $G_\mathrm{II}$ on both forms of the Darboux equation are discussed in \S\ref{tran3}, while the symmetric group action on the accessory parameter $h$ is discussed in \S\ref{S:accessory}. In \S\ref{S:expansion}, we give the convergence criteria of local series expansions about the regular singularities on the torus. Then we discuss the criteria that lead to the termination of these infinite expansions. These finite sum solutions which we call Darboux polynomials are analogous to the classical Lam\'e polynomials.  Section \ref{S:192} enumerates the 192 local series expansions. We apply our main results to well-known special cases of the Darboux equation in Section \ref{S:Special}.

\section{Group actions: transpositions of local solutions}\label{S:g1}
As we have mentioned in the Introduction that the symmetry group of the Darboux equation is given by a semi-direct product $G_\mathrm{I}\rtimes_\Gamma G_\mathrm{II}$.
In this section, we introduce $G_\mathrm{I}$ which consists of \textit{transpositions} of two local solutions at each of the four regular singularities. As one can easily see below, these actions are automorphisms of the Darboux equations which leave the equation unchanged. Specifically, these are given by the transformations on the parameters $\xi,\,\eta,\,\mu,\,\nu$, which represent the local monodromies,  in the Darboux equations in both the Jacobian form (\ref{E:darboux}) and in the Weierstrass form (\ref{E:darbouxW}).


Let $\gamma\in \{\xi,\,\eta,\,\mu,\,\nu\}$. Denote $\gamma^+=\gamma$ and $\gamma^-=-\gamma-1$. Since $$\gamma^+(\gamma^++1)=\gamma^-(\gamma^-+1)=\gamma(\gamma+1),$$
both equations (\ref{E:darboux}) and (\ref{E:darbouxW}) remain unchanged under the following transformations
$$\xi\mapsto\xi^\pm,\ \eta\mapsto\eta^\pm,\ \mu\mapsto\mu^\pm,\ \nu\mapsto\nu^\pm.$$
%
\medskip

Such transformations give the following
	\begin{definition}
		Let $G_\mathrm{I}=\{+,-\}^4$ be the group of $\mathrm{16}$ automorphisms of
the set of parameters in the Darboux equation defined as above.	This group can also be identified as the (additive) group of mappings from $\{0,\,1,\,2,\, 3\}$ to $\Z_2$.
\end{definition}

It is easy to see that $G_\mathrm{I}$ forms a group which is isomorphic to $(\Z_2)^4$.

\bigskip

\section{Group actions on half-periods: Weierstrass elliptic form}\label{tran1}

In this section, we consider $G_\mathrm{II}$ in the semi-direct product $G_\mathrm{I}\rtimes_\Gamma G_\mathrm{II}$ which consists of translations and permutations of the half-periods of the underlying torus of the Darboux equation in the Weierstrass form. The consideration of Darboux equation in the Jacobian form will be discussed in the next section. So we discuss the transformations of Weierstrass elliptic functions, and hence they give transformations from a Darboux operator to another Darboux operator.

In general, if $w:X\to\mathbb{C}$ is a function and $g\in \mathrm{Aut}(X)$, then $g$ acts on $w$ naturally by
composition. In the study of Darboux equations, the elliptic functions like $\wp(u,\tau)$ or $\sn^2(u,k)$
are involved. The first task is to clarify the domain of these elliptic functions so that the automorphisms
of their domain act on these elliptic functions by composition (or pulling back). The natural candidate
is the ``universal family" of elliptic curves. However, it is standard that this universal family does not
exist \cite[p. 35]{HM}. Therefore, we will settle with a ``six-to-one" family.

The first step in our general formulation is to parametrize ``lattices with an ordered basis".
Recall that if a given lattice $\Lambda$ in $\mathbb{C}$ is equipped with an ordered base
$(2\omega_1,\, 2\omega_3)$. Its $\lambda$-invariant is defined by
$$
\lambda(\omega_1,\,\omega_3)=\dfrac{\wp(\omega_3)-\wp(\omega_1+\omega_3)}
{\wp(\omega_1)-\wp(\omega_1+\omega_3)}.
$$
It can be shown that this $\lambda$-\textit{invariant} is invariant under the action of
$\Gamma(2)=\{M\in SL(2,\Z):M=I\mbox{ (mod 2) }\}$ \cite[p. 9]{Lehner1964} and \cite[p. 278]{ahlfors}. Therefore, ``lattices with an ordered base" are parametrized by the well-known  quotient $\mathbb{H}/\Gamma(2)=\mathrm{X}(2)$ \cite[p. 278]{ahlfors}.

Our next step is to parametrize all the lattices. We know that two ``lattices with an ordered basis"
$(2\omega_1,\, 2\omega_3)$ and $(2\omega'_1,\, 2\omega'_3)$ have isomorphic underlying lattices when there exists $M\in SL(2,\C)$ which takes $\omega_1$ to $\omega'_1$ and takes $\omega_3$ to $\omega'_3$.
This $\mathbb{Z}-$linear map from span$(2\omega_1,\,2\omega_3)$ to span$(2M\omega_1,\, 2M\omega_3)$ induces
an action on the $\lambda$-invariants.

It is well-known that the orbit of the $\lambda$-invariant under the full modular group $\Gamma=\textrm{SL}(2,\,\Z)$ is
$$
\Big\{\lambda,\,\dfrac{1}{1-\lambda},\,\dfrac{\lambda-1}{\lambda},\,
\dfrac{1}{\lambda},\,\dfrac{\lambda}{\lambda-1},\,1-\lambda\Big\}.
$$
These are in fact the M\"obius transformations which are permutations of $\{0,1,\infty\}$.  These M\"obius transformations form a group known
as the \textit{anharmonic group}  anh \cite[p. 278]{ahlfors}. Indeed, we have
	\[
		\mathrm{anh}= \big\{M\in SL(2,\, \mathbb{Z}): M:\ \mathbb{CP}^1\backslash\{0,\, 1,\, \infty\}\to  \mathbb{CP}^1\backslash\{0,\, 1,\, \infty\}\big\}.
	\]

Since  $\Gamma(2)$ is a normal subgroup of $\Gamma$ \cite[p. 9]{Lehner1964}, so $\textrm{anh}$ is 
isomorphic to the quotient $\Gamma/\Gamma(2)$. 
Thus lattices are parametrized by $\mathrm{X}(2)$ modulo the action of anh.

The corresponding transformations of lattices with ordered basis are named, following the classical convention \cite[p. 369]{EAM2}, as
$X=I$, $A$, $B$, $C$, $D$, $E$. They also correspond to the action of permutations $\rho_X$ of $e_i=\wp(\omega_i)$ on the
$\lambda$-invariant. These are summarized in the Table \ref{TT} below.
\medskip

\begin{table}
\caption{}\label{TT}

\begin{center}


\begin{tabular}{c||c|c|c}
$X$ & {Six cross-ratios} & {Matrix representative} & $\rho_X$ \\
\hline
$I$ &$\lambda$ &$M_I=\begin{bmatrix}1&0\\0&1\end{bmatrix}$ &$(e_1)(e_2)(e_3)$\\
\hline
$A$ &$\dfrac{\lambda}{\lambda-1}$ &$M_A=\begin{bmatrix}1&0\\1&1\end{bmatrix}$ &$(e_1)(e_2e_3)$\\
\hline
$B$ &$\dfrac{1}{\lambda}$ &$M_B=\begin{bmatrix}0&1\\-1&0\end{bmatrix}$ &$(e_1e_3)(e_2)$\\
\hline
$C$ &$1-\lambda$ &$M_C=\begin{bmatrix}1&1\\0&1\end{bmatrix}$ &$(e_1e_2)(e_3)$\\
\hline
$D$ &$\dfrac{\lambda-1}{\lambda}$ &$M_D=\begin{bmatrix}-1&1\\-1&0\end{bmatrix}$ &$(e_1e_2e_3)$\\
\hline
$E$ &$\dfrac{1}{1-\lambda}$ &$M_E=\begin{bmatrix}0&1\\-1&-1\end{bmatrix}$ &$(e_1e_3e_2)$
\end{tabular}
\end{center}
\end{table}
\medskip

Now we consider a family of elliptic curves whose fibre at $\tau$ (together with the point at infinity) is cut
out (i.e., zero-locus) by
$$
y^2=4x^3-g_2(\tau)x-g_3(\tau).
$$
Explicitly, we have a family $\mathbb{E}$ (together with the point at infinity, i.e., in $\mathbb{CP}^2$)
$$
\mathbb{E}=\{(x,y;\tau)\in\C^2\times \mathrm{X}(2):y^2=4x^3-g_2(\tau)x-g_3(\tau)\}
\stackrel{\lambda\circ \pi_3}{\longrightarrow }\mathbb{C}\backslash\{0,1\},
$$

\medskip

\noindent where the map $\pi_3$ denotes the projection to the third coordinate.
The Weierstrass elliptic function $\wp$ can now be regarded as a function defined on $\mathbb{E}$, which
is just the projection to the first coordinate.

In fact, each of such six transformations $I$, $A$, $B$, $C$, $D$, $E:\mathbb{E}\to\mathbb{E}$ can be
described explicitly by
	\begin{align}\label{E:anh-weierstrass}
		I(x,\, y;\, \tau) & =(x,\, y;\, \tau); \nonumber \\
		A(x,\,y;\, \tau) & =\big((\tau-1)^2 x,\,(\tau-1)^3 y;\, \tau/(\tau-1)\big); \nonumber \\
		B(x,\,y;\, \tau) & =\big(\tau^2 x,\, \tau^3 y;\, 1/\tau\big); \nonumber \\
		C(x,\,y;\, \tau) & =\big(x,\, y;\, 1-\tau\big);\\
		D(x,\, y;\, \tau) & =\big(\tau^2 x,\,\tau^3 y;\, (\tau-1)/\tau\big);\nonumber \\
		E(x,\, y;\, \tau) &= \big((1-\tau)^2 x,\,(1-\tau)^3 y;\, 1/(1-\tau)\big). \nonumber
	\end{align}

It is easy to verify that they are well-defined maps from $\mathbb{E}$ to $\mathbb{E}$, since $g_2$ and $g_3$ are modular forms of weights 4 and 6 respectively.
On the other hand, given a torus with ordered base $(2\omega_1,\, 2\omega_3)$, the subgroup of the order 2 points
of the corresponding torus would be
$$
T=\{0,\, \omega_1,\, \omega_3,\ \omega_2=\omega_1+\omega_3\}.
$$
For each $j\in\{1,\,2,\, 3\}$, let $I_j$ be an action on $T$ which is translation by $\omega_j$. Obviously,
$K=\{I_0,\, I_1,\, I_2,\, I_3\}$ is isomorphic to the Klein four-group.
\medskip

Finally we are able to describe the group  $G_\mathrm{II}$:

\begin{definition} The group $G_\mathrm{II}$  consists of commutative diagrams of the form
		\begin{equation}\label{E:communtative}
\begin{array}{ccc}
\mathbb{E}&\longrightarrow  &\mathbb{E}\\
\lambda\circ \pi_3\left\downarrow\rule{0cm}{0.4cm}\right.\phantom{\lambda\circ\pi_3}
& &\phantom{\lambda\circ\pi_3}\left\downarrow\rule{0cm}{0.4cm}\right.\lambda\circ\pi_3\\
\mathbb{C}\backslash\{0,1\}& \longrightarrow & \mathbb{C}\backslash\{0,1\}
\end{array}
    \end{equation}
so that both of the horizontal maps are biholomorphic. The group multiplication between two such commutative diagrams is defined by obvious concatenation of two such diagrams.
\end{definition}

\begin{theorem}\label{E:lower} Let $K,\, G_\mathrm{II}$ and $\mathrm{anh}$ be the groups defined above. Then we have the short exact sequence
	$$
0\longrightarrow\  K\  \longrightarrow\  G_\mathrm{II}\longrightarrow \ \mathrm{anh}\longrightarrow\  0.
$$
\end{theorem}
\medskip

\begin{proof}
 Here the surjective map $G_\mathrm{II}\ \longrightarrow \ \mathrm{anh}$ is canonical in the sense that we apply the lower homomorphism of the above commutative diagram (\ref{E:communtative}). We have established the right half of the short exact sequence. To establish the left-half of the short exact sequence, we define the map $K\to G_\mathrm{II}$ as follows.  Recall that the cubic
curve $y^2=4x^3-g_2(\tau)x-g_3(\tau)$ is an Abelian group (with the usual group addition ``$+$" of two points of the elliptic curve is defined by the inverse point of a third point obtained from the intersection of the straight line passing through the two points and the elliptic curve)
with the point at infinity chosen as the identity. Let $x_1$, $x_2$, $x_3$ be the zeros of $\wp^\prime(\cdot \, ;\, \tau)$. Then the map $K\to G_\mathrm{II}$ given by
	\begin{align}\label{E:half-period-shifts}
		I_0 &\longmapsto [(x,\, y;\, \tau)\mapsto (x,\,y;\, \tau)],\nonumber \\
		I_1 &\longmapsto [(x,\,y;\, \tau)\mapsto (x,\, y;\, \tau)+(\wp(x_1),\, 0;\, \tau)],\nonumber \\
		I_2 &\longmapsto [(x,\, y;\, \tau)\mapsto (x,\, y;\,\tau)+(\wp(x_2),\, 0;\, \tau)],\\
		I_3 &\longmapsto [(x,\, y;\, \tau)\mapsto (x,\, y;\, \tau)+(\wp(x_3),\, 0;\, \tau)] \nonumber
	\end{align}
is clearly injective.
\end{proof}
\medskip

\begin{remark}\label{R:permutation}
We have seen that the anharmonic group anh corresponds to the permutation group of the half-periods $\{\omega_1,\omega_2,\omega_3\}$, while the Klein four group acts on a torus by translations by half-periods. Therefore they generate a group $G_{\mathrm{II}}$ which corresponds to the permutation group of the order-two points $\{0,\omega_1,\omega_2,\omega_3\}$.
\end{remark}
\medskip

\subsection{Maier's hyper-Kummer group $\mathfrak{K}$ and $G_\mathrm{II}$} The analysis of $G_\mathrm{II}$ can also be applied to the hyper-Kummer group mentioned in Maier {\cite{Maier1,Maier}}. Recall that the hyper-Kummer group $\mathfrak{K}$ (in the case of $\mathbb{CP}^1$ together with four marked points)
 is the subgroup of $S_4\times \mathrm{PSL}(2,\, \C)$ which acts on $(0,\,1,\,\infty,\,\lambda)$ leaving each of the first
three coordinates fixed. We define a map
\smallskip

\begin{equation}\label{E:maier}
\mathfrak{K}\longrightarrow\Big\{\lambda,\,\dfrac{1}{1-\lambda},\,\dfrac{\lambda-1}{\lambda},\,
\dfrac{1}{\lambda},\,\dfrac{\lambda}{\lambda-1},\,1-\lambda\Big\}
\end{equation}
\smallskip

\noindent so that $(\sigma,\, T)\in S_4\times \mathrm{PSL}(2,\, \C)$ is sent to the cross ratio
$(\sigma(0),\,\sigma(1),\,\sigma(\infty),\, \sigma(\lambda))$. Note that the latter group of order six is
identified to the anharmonic group $\mathrm{anh}$. We remark that one of the $\sigma\in S_4$  or $T\in\mathrm{PSL}(2,\, \C)$ is redundant in the above description since $\sigma$ uniquely determines $T$ (and vice versa). However, no such simplification is possible when the discussion is about a second order differential equation with $n$ singular points where $n\ge 5$ as Maier considered  {\cite{Maier1}}.

\medskip

\begin{theorem}\label{ExSeq} We have the following commutative diagram consisting of two rows of short exact sequences:
	\begin{equation}\label{E:upper-lower}
\begin{array}{rcccl}
0\to &K&\to\mathfrak{K}\to&\mathrm{anh}&\to 0\\
 &||&\downarrow&||& \\
0\to &K&\to G_\mathrm{II}\to&\mathrm{anh}&\to 0.
\end{array}
	\end{equation}
In particular, the hyper-Kummer group $\mathfrak{K}$ is isomorphic to $G_\mathrm{II}$.
\end{theorem}
\medskip

\begin{proof} We have already established the lower short exact sequence in the Theorem \ref{E:lower}. To establish the upper short exact sequence
$$
0 \longrightarrow  K \longrightarrow \mathfrak{K}\longrightarrow  \mathrm{anh}\longrightarrow  0,
$$
we note that the kernel of the map $\mathfrak{K}\to \mathrm{anh}$ as defined in (\ref{E:maier}) above consists of wo disjoint pairs of transpositions of the points in the cross-ratios $(\sigma(0),\,\sigma(1),\,\sigma(\infty),\, \sigma(\lambda))$ which can easily be seen to be isomorphic to the Klein four-group $K$.

Now we would like to construct a map $\mathfrak{K}\to G_\mathrm{II}$. First of all, we note that the family
$\mathbb{E}\to\mathbb{C}\backslash\{0,1\}$ is written explicitly as
\smallskip

$$
\mathbb{E}=\big\{(x,\, y;\,\tau)\in\mathbb{C}^2\times \mathrm{X}(2):y^2=4x^3-g_2(\tau)\, x-g_3(\tau)\big\}.
$$
\smallskip

\noindent Then $(\sigma,\left(\begin{array}{cc}a&b\\c&d\end{array}\right))\in S_4\times PSL(2,\mathbb{C})$ induces the
commutative diagram
$$
\begin{array}{ccc}
\mathbb{E}&\longrightarrow&\mathbb{E}\\
\left\downarrow\rule{0cm}{0.4cm}\right.&&\left\downarrow\rule{0cm}{0.4cm}\right.\\
\mathbb{C}\backslash\{0,1\}&\longrightarrow&\mathbb{C}\backslash\{0,1\}
\end{array}
$$
where the map in the top row is given by
$$
(x,\, y;\, \tau)\mapsto \Big((c\tau +d)^2x,\, (c\tau+d)^3y;\  \frac{a\tau+b}{c\tau+d}\Big).
$$
The other maps in the diagram (\ref{E:upper-lower}) are natural.
\end{proof}
\smallskip

\noindent The commutative diagram (\ref{E:upper-lower}) established above describes that the hyper-Kummer group $\mathfrak{K}$ of Heun equation has the same kind of semi-direct product decomposition as $G_\mathrm{II}$.

\subsection{The action of $G_\mathrm{II}$ on Darboux operator in Weierstrass form}\label{G2W}

Given an elliptic function $\wp:\mathbb{E}\to\mathbb{CP}^1$, the group $G_\mathrm{II}$ acts on $\wp$ by
pulling back. This action by anh can be found in (\ref{E:anh-weierstrass}).

For each $\tau$ in $\mathbb{H}$, let $e_1(\tau)$, $e_2(\tau)$, $e_3(\tau)$ be the zeros of
\[
4x^3-g_2(\tau)x-g_3(\tau).
\]
\begin{definition}
The Darboux operator is the differential operator on $\mathbb{E}$
\[
\begin{array}{rl}
D_{(\tau; \xi, \eta, \mu, \nu)}=&\Big(y\dfrac{d}{dx}\Big)^2-\Big[\xi(\xi+1)x\\
 &+\eta(\eta+1)\dfrac{(x-e_2(\tau))(x-e_3(\tau))-(x-e_1(\tau))^2}{x-e_1(\tau)}\\
 &+\mu(\mu+1)\dfrac{(x-e_3(\tau))(x-e_1(\tau))-(x-e_2(\tau))^2}{x-e_2(\tau)}\\
 &+\nu(\nu+1)\dfrac{(x-e_1(\tau))(x-e_2(\tau))-(x-e_3(\tau))^2}{x-e_3(\tau)}\Big].
\end{array}
\]
\end{definition}

Note that $e_1(\tau)$, $e_2(\tau)$, $e_3(\tau)$ are formally modular forms of weight 2, in particular, for $j=1,2,3$ and
$X=I,A,B,C,D,E$,
	\begin{equation}\label{E:ej}
		e_j\Big(\frac{a\tau+b}{c\tau+d}\Big)=(c\tau+d)^{2}\rho_X(e_j)(\tau),
	\end{equation}
where $M_X=\left[\begin{array}{cc}a&b\\c&d\end{array}\right]$ and $\rho_X$ are given in Table \ref{TT}.
 The actions of
anh (and hence $G_\mathrm{II}$) on the Darboux operator are easily obtained as the following theorem by direct computation and (\ref{E:ej}).
\begin{theorem}
For each $M_X=\left[\begin{array}{cc}a&b\\c&d\end{array}\right]\in$ anh, the pull-back of the Darboux operator
$D_{(\tau; \xi, \rho_X(\eta), \rho_X(\mu), \rho_X(\nu))}$ by $M_X$ is
\[
(c\tau+d)^{-2}D_{(\frac{a\tau+b}{c\tau+d};\, \xi, \eta, \mu, \nu)}.
\]
\end{theorem}

Now we look at the action of anh (and hence $G_\mathrm{II}$) on Darboux equation (\ref{E:darbouxW}) instead of the Darboux operator.
For example, we consider the map $A$ from~(\ref{E:anh-weierstrass}) acting on the Darboux equation (\ref{E:darbouxW}) with
$\omega_1=\tau/2$, $\omega_2=1/2+\tau/2$ and $\omega_3=1/2$. Notice that
\begin{eqnarray*}
A(\wp(u;\ \tau),\ \wp'(u;\ \tau);\tau)&=&\big((\tau-1)^2\wp(u;\ \tau),\ (\tau-1)^3\wp'(u;\ \tau);\frac{\tau}{\tau-1}\big)\\
                                      &=&(\wp(\tilde{u};\ \tilde\tau),\ \wp'(\tilde{u};\ \tilde{\tau});\tilde\tau),
\end{eqnarray*}
where $\tilde{u}=\frac{u}{\tau-1}$ and $\tilde{\tau}=\frac{\tau}{\tau-1}$. So the Darboux equation (\ref{E:darbouxW}) becomes
\beqn\label{E:newdarbouxW}
&&(\tau-1)^{-2}\frac{{d}^{2}y}{{d\tilde u}^{2}}+\Big[h-(\tau-1)^{-2}\big({\xi(\xi+1)}{\ts\wp(\tilde u;\tilde\tau)}
-{\ts\eta(\eta+1)\wp(\tilde u+\frac{\tilde\tau}{2};\tilde\tau)}\nonumber\\&&\ \hspace{2cm} -
{\ts\nu(\nu+1)\wp(\tilde u+\frac{\tilde\tau}{2}+\frac{1}{2};\tilde\tau)}-\mu(\mu+1){\wp(\tilde u+\frac{1}{2};\tilde\tau)\big)}\Big]y=0,\nonumber\\
\eeqn
which is equivalent to
\beq
&&\frac{{d}^{2}y}{{d\tilde u}^{2}}+\Big[\tilde h-{\xi(\xi+1)}{\ts\wp(\tilde u;\tilde\tau)}
-{\ts\eta(\eta+1)\wp(\tilde u+\frac{\tilde\tau}{2};\tilde\tau)}\\&&\ \hspace{2cm} -
{\ts\nu(\nu+1)\wp(\tilde u+\frac{\tilde\tau}{2}+\frac{1}{2};\tilde\tau)}-\mu(\mu+1){\wp(\tilde u+\frac{1}{2};\tilde\tau)}\Big]y=0,
\eeq
where $\tilde{h}=(\tau-1)^2h.$
Similarly, one can write down the group action of $G_\mathrm{II}$ on the Darboux equation in Weierstrass form by other elements $B,\, C,\, D,\, E$, while the action of the Klein four-group on the Darboux equation consists simply of shifting the half-periods in the independent variable $u$.

\section{Group actions on half-periods: Jacobi elliptic form}\label{tran2}

Since the Jacobi elliptic functions are much more convenient to compute, so historically, most of these classical differential equations with doubly periodic coefficients, such as the Darboux equation, and those listed in the Appendix \ref{S:Sparre} are written in Jacobian forms \cite{BF}. However, as we have explained earlier in the last section that
the action of the group $G_\mathrm{II}$ on Jacobi elliptic functions is much more complicated compared with that of the Weierstrass elliptic functions.

 This is due to the parametrization of the elliptic curves using Jacobi elliptic functions. For example, the $\sn (u,\, k)$ depends on the independent variable $u$ and modulus $k$, and the half-periods are implicit in the notation. One can convert the half-periods $\omega_1,\, \omega_2,\, \omega_3$ in terms of the modulus $k$ and the complementary modulus $k^\prime=\sqrt{1-k^2}$. We note that $k^2$ is the $\lambda$-invariant of the torus. We refer the reader to the Appendix \ref{S:conversion} for the relationships. The lower short exact sequence of (\ref{E:upper-lower})  in Theorem \ref{ExSeq} gives rise to symmetry group $G_\mathrm{II}$ whose actions on the Jacobi elliptic functions
$\sn,\ \cn,\ \dn$ can also be split into two parts, that are, the actions by the Klein four-group $K$ and the actions of anharmonic group $\mathrm{anh}$.

\subsection{The action of anh on Jacobi elliptic functions}\label{G2D}
In this section, we discuss the transformations $(u,\,k)\mapsto(\tau(u,\, k),\,\kappa(k))$ of Jacobi elliptic functions into other Jacobi elliptic functions by the anharmonic group anh. The six permutations  $I$,  $A$, $B$, $C$, $D$, $E$ written in the Table \ref{T0} on the Jacobi elliptic functions are different from those given in (\ref{E:anh-weierstrass}) and are more complicated. Indeed, they had been already given in  \cite[p. 369]{EAM2} which we reproduce in Table \ref{T0}.

{\scriptsize
\begin{table}
\caption{Transformations of Jacobi elliptic functions}\label{T0}

\begin{center}
%
\begin{tabular}{c||c|c|c||c|c||c|c|c}
  ${X}$ & ${\tau_{X}}$ & $\kappa_X$ & $\kappa_X'$ & $K(\kappa_X)$ & $K'(\kappa_X)$ & $\sn(\tau_{X},\kappa_X)$ &
  $\cn(\tau_{X},\kappa_X)$ & $\dn(\tau_{X},\kappa_X)$\\
  \hline
  $I$ & $u$ & $k$ & $k'$ & $K(k)$ & $K'(k)$& $\sn(u,k)$ & $\cn(u,k)$ & $\dn(u,k)$\\
  \hline
  $A$ & $k'u$ & $ik{k'}^{-1}$ & ${k'}^{-1}$ & $k'K(k)$ & $k'[K'(k)-iK(k)]$ &$k'\sd(u,k)$ & $\cd(u,k)$ & $\nd(u,k)$\\
  \hline
  $B$ & $-iu$ & $k'$ & $k$ & $K'(k)$ & $K(k)$ &$i\sc(u,k)$ & $\nc(u,k)$ & $\dc(u,k)$\\
  \hline
  $C$ & $ku$ & $k^{-1}$ & $-ik'{k}^{-1}$ & $k[K(k)+iK'(k)]$ &$kK'(k)$ &$k\sn(u,k)$ & $\dn(u,k)$ & $\cn(u,k)$\\
  \hline
  $D$ & $-ik'u$ & ${k'}^{-1}$ &$-ik{k'}^{-1}$ &$k'[K'(k)+iK'(k)]$ &$k'K(k)$ &$-ik'\sc(u,k)$ & $\dc(u,k)$ & $\nc(u,k)$\\
  \hline
  $E$ & $-iku$ &$ik'{k}^{-1}$ &$k^{-1}$ & $kK'(k)$ & $k[K(k)+iK'(k)]$ &$-ik\sd(u,k)$ & $\nd(u,k)$ & $\cd(u,k)$\\

\end{tabular}
\end{center}
\end{table}
}


{\small
%
\begin{table}
\caption{}\label{T1}

\begin{tabular}{c||c|c|c||c|c|c}
  ${X_i}$ & ${\tau_{X_i}}$ & $\kappa_X$ &$ \kappa'_X$ & $\sn(\tau_{X_i},\kappa_X)$ & $\cn(\tau_{X_i},\kappa_X)$ & $\dn(\tau_{X_i},\kappa_X)$\\
  \hline
  $I_0$ & $u$ & & &$\sn(u,k)$ & $\cn(u,k)$ & $\dn(u,k)$\\
  $I_1$ & $u+K(k)$ & & &$\cd(u,k)$ & $-k'\sd(u,k)$ & $k'\nd(u,k)$\\
  $I_2$ & $u+K(k)+iK'(k)$ & $k$& $k'$ &$k^{-1}\dc(u,k)$ & $ik'k^{-1}\nc(u,k)$ & $ik'\sc(u,k)$\\
  $I_3$ & $u+iK'(k)$ & & &$k^{-1}\ns(u,k)$ & $-ik^{-1}\ds(u,k)$ & $-i \cs(u,k)$\\ \hline
  $A_0$ & $k'u$ & & &$k'\sd(u,k)$ & $\cd(u,k)$ & $\nd(u,k)$\\
  $A_1$ & $k'(u+K(k))$ & & &$\cn(u,k)$ & $-\sn(u,k)$ & ${k'}^{-1}\dn(u,k)$\\
  $A_2$ & $k'(u+K(k)+iK'(k))$ & $ikk'^{-1}$&${k'}^{-1}$ &$-ik^{-1}\ds(u,k)$ & $k^{-1}\ns(u,k)$ & $-i{k'}^{-1}\cs(u,k)$\\
  $A_3$ & $k'(u+iK'(k))$ & & &$ik'k^{-1}\nc(u,k)$ & $k^{-1}\dc(u,k)$ & $i\sc(u,k)$\\ \hline
  $B_0$ & $-iu$ & & &$i\sc(u,k)$ & $\nc(u,k)$ & $\dc(u,k)$\\
  $B_1$ & $-i(u+K(k))$ & & &$i{k'}^{-1}\cs(u,k)$ & ${k'}^{-1}\ds(u,k)$ & $-\ns(u,k)$\\
  $B_2$ & $-i(u+K(k)+iK'(k))$ & $k'$ & $k$ &${k'}^{-1}\dn(u,k)$ & $ik{k'}^{-1}\cn(u,k)$ & $k\sn(u,k)$\\
  $B_3$ & $-i(u+iK'(k))$ & & &$\nd(u,k)$ & $ik\sd(u,k)$ & $k\cd(u,k)$\\ \hline
  $C_0$ & $ku$ & & &$k\sn(u,k)$ & $\dn(u,k)$ & $\cn(u,k)$\\
  $C_1$ & $k(u+K(k))$ & & &$k\cd(u,k)$ & ${k'}\nd(u,k)$ & $-k'\sd(u,k)$\\
  $C_2$ & $k(u+K(k)+iK'(k))$ & $k^{-1}$&$-ik'k^{-1}$ &$\dc(u,k)$ & $ik'\sc(u,k)$ & $ik'{k}^{-1}\nc(u,k)$\\
  $C_3$ & $k(u+iK'(k))$ & & &$\ns(u,k)$ & $-i\cs(u,k)$ & $-ik^{-1}\dc(u,k)$\\ \hline
  $D_0$ & $-ik'u$ & & &$-ik'\sc(u,k)$ & $\dc(u,k)$ & $\nc(u,k)$\\
  $D_1$ & $-ik'(u+K(k)+iK'(k))$ & & &$-i\cs(u,k)$ & $-\ns(u,k)$ & $-{k'}^{-1}\ds(u,k)$\\
  $D_2$ & $-ik'(u+K(k)+iK'(k))$ & $k'^{-1}$& $-ikk'^{-1}$ &$\dn(u,k)$ & $k\sn(u,k)$ & $ik{k'}^{-1}\cn(u,k)$\\
  $D_3$ & $-ik'(u+iK'(k))$ & & &$k'\nd(u,k)$ & $k\cd(u,k)$ & $ik\sd(u,k)$\\ \hline
  $E_0$ & $-iku$ & & &$-ik\sd(u,k)$ & $\nd(u,k)$ & $\cd(u,k)$\\
  $E_1$ & $-ik(u+K(k))$ & & &$-ik{k'}^{-1}\cs(u,k)$ & ${k'}^{-1}\dn(u,k)$ & $-\sn(u,k)$\\
  $E_2$ & $-ik(u+K(k)+iK'(k))$ & $-ik'k^{-1}$& $k^{-1}$ &$-i{k'}^{-1}\ds(u,k)$ & $-i{k'}^{-1}\cs(u,k)$ & ${k}^{-1}\ns(u,k)$\\
  $E_3$ & $-ik(u+iK'(k))$ & & &$\nc(u,k)$ & $i\sc(u,k)$ & ${k}^{-1}\dc(u,k)$\\
  \end{tabular}
\end{table}
\medskip
}

The second column $\tau_X$ denotes the  change of the independent variable $u$
under the actions of $X$. The second and third columns denote change in the modulus $\kappa_X$ and the complementary modulus $\kappa^\prime_X$, respectively,  under the actions of $X$. The fourth and the fifth columns denote the change of the quarter-periods $K$ and $K^\prime$ under the actions of $X$, while the last three columns describe the transformations on the Jacobi elliptic functions $\sn,\, \cn,\,\dn$. We note that, as far as the Darboux equation in the Jacobian form is concerned, the periodic coefficients, written in terms of $\sn^2,\, \cn^2,\, \dn^2$,  have half-periods $K$ and $iK^\prime$ as denoted in the fourth and fifth  columns in the Table \ref{T0}.
\medskip

\subsection{The action of the Klein four-group $K$ on Jacobi elliptic functions}

Similarly, the action of the Klein four-group $K$ on the Jacobi elliptic functions $\sn,\, \cn,\, \dn$ and their moduli are given by the first section (i.e., $I_0,\, I_1,\, I_2,\, I_3$) of the Table \ref{T1}. These actions correspond to the \textit{shifting} of the half-periods of $\sn^2,\, \cn^2,\, \dn^2$ of the Jacobian form of Darboux equation (\ref{E:darboux}). We will apply these transformations when we discuss the joint group actions of the Darboux equations in the next section (\S\ref{joint-action}). We remark that as one can see that the transformations are much more complicated compared with those for the Weierstrass form of the Darboux equation.

\subsection{The joint actions of the Klein four-group $K$ and anh}\label{joint-action}
Finally, we come to describing the joint group actions of the Klein four-group $K$ and the anharmonic group anh. The combined actions act on the Jacobi elliptic functions are given by the Table~\ref{T1} which consists of six types of permutations of quarter-periods (one period of $\sn$ doubles that of $\sn^2$)
 each of which leaves the origin of the torus unchanged, and four types of translations of the quarter-periods. This is reflected in the Table \ref{T1} having six horizontal sections labelled by type $I,\, A, \, B,\, C,\, D,\, E$ which corresponds to actions of anh (permutations of quarter-periods), and each of which is further subdivided into four translations. These four translations represent the Klein four-group actions. For example, the map $B_3$ represents the transformation by the translation $I_3$ acting after the action of $B$.

\section{Joint actions of $G_\mathrm{I}$ and $G_\mathrm{II}$ automorphisms}\label{tran3}

\begin{definition} Let $G$ be the symmetry group of the Darboux equation generated by the groups $G_\mathrm{I}$ (i.e., the transpositions of local solutions) and $G_\mathrm{II}$ (i.e., actions on half-periods).
\end{definition}

\begin{theorem} The group $G$ is a semi-direct product \footnote{See the Definition \ref{semidirect} in Appendix \ref{SD} for the semi-direct product.}
	\[
		G\cong G_\mathrm{I}\rtimes_\Gamma G_\mathrm{II},
	\]
for some action $\Gamma: G_\mathrm{II}\to \mathrm{Aut}(G_I)$.
\end{theorem}

\begin{proof}

Identify $G_\mathrm{I}$ as the group of mappings $\{0,1,2,3\}\to\Z_2$, and identify $G_\mathrm{II}$ as the group of permutations of half-periods. That is, the permutations group of $\{0,1,2,3\}$ (see Remark \ref{R:permutation}). Then $G_\mathrm{II}$ acts on $G_\mathrm{I}$ from the left and we call this action $\Gamma:G_\mathrm{II}\to\mathrm{Aut}(G_\mathrm{I})$. It remains to show that for each $s\in G_\mathrm{I}$ and $X\in G_\mathrm{II}$,
\[
\Gamma(X)(s)=X\,s\,X^{-1}.
\]
But this verification is routine.
\end{proof}

\begin{remark}
Since $G_\mathrm{II}$ is a permutation group, the group $G$ above is indeed the group of \textit{signed permutations}. It can also be described as the Coxeter group $B_4$.
\end{remark}






\section{Group actions on accessory parameters}\label{S:accessory}
Under the combined actions of $G\cong G_\mathrm{I}\rtimes_\Gamma G_\mathrm{II}$, we have
\begin{theorem}  For each $X=I,\, A,\, B,\, C,\, D,\, E$ ; $i=0,1,\, 2,\, 3$, $X_i$ transforms the Darboux equation (\ref{E:darboux}) to the form
{\beqn\label{darboux1}
		&&\frac{{d}^{2}y}{{dw}^{2}}+\Big[h_X-{\sigma_{X_i}(\xi)(\sigma_{X_i}(\xi)+1)}{\ts\ns^2(w,\kappa_X)}-
{\ts\sigma_{X_i}(\eta)(\sigma_{X_i}(\eta)+1)\dc^2(w,\kappa_X)}\nonumber\\&&\ -
{\ts\sigma_{X_i}(\mu)(\sigma_{X_i}(\mu)+1)\kappa_X^2\cd^2(w,\kappa_X)}-\sigma_{X_i}(\nu)(\sigma_{X_i}(\nu)+1)\kappa_X^2{\ts\sn^2(w,\kappa_X)}\Big]y=0,\nonumber\\
\eeqn}
with the Riemann $P$-scheme
\[
	P_{\mathbb{C}\slash\Lambda}\begin{Bmatrix}
\ 0\ &\ K(\kappa_X)\ &\ K(\kappa_X)+iK'(\kappa_X)\ &\ iK'(\kappa_X)\ &\\
\sigma_{X_i}(\xi)+1&\sigma_{X_i}(\eta)+1&\sigma_{X_i}(\mu)+1&\sigma_{X_i}(\nu)+1&w;\, h_X\\
-\sigma_{X_i}(\xi)&-\sigma_{X_i}(\eta)&-\sigma_{X_i}(\mu)&-\sigma_{X_i}(\nu)&
\end{Bmatrix},
	\]
\noindent for some $h_X$, where $w:=\tau_{X_i}(u,k)$ and $\kappa_X$ are listed in Table \ref{T1}, and $\sigma_{X_i}$ is a permutation of $\{\xi,\eta,\mu,\nu\}$ as listed in Table \ref{T2}.
\end{theorem}

\begin{remark}
The group $G_\mathrm{I}$ leaves the Darboux equation unchanged.
\end{remark}
\begin{remark}
The actions of Klein four-group $K$ correspond to shifting of the half-periods of the torus, hence resulting in permuting the four Jacobi elliptic functions in the Darboux potential, it obviously leaves the accessory parameter $h_{X_i}$ unaltered, so we can denote the accessory parameter by $h_X$, instead of $h_{X_i}$.
\end{remark}


The description of $h_X$ for each transformation $X_i$ can be found in Table \ref{T2}.
One can see from Table \ref{T2} that these new accessory parameters can be divided into six types.

{\footnotesize
\begin{table}
\caption{}\label{T2}
\begin{center}

\small
\begin{tabular}{c||c||c|c|c|c}
  ${X}$ & $h_X$ & $\sigma_{X_0}$ & $\sigma_{X_1}$ & $\sigma_{X_2}$ & $\sigma_{X_3}$\\
  \hline
  $I$ & $h$ & $(\xi)(\eta)(\mu)(\nu)$ & $(\xi\eta)(\mu\nu)$ & $(\xi\mu)(\eta\nu)$ & $(\xi\nu)(\eta\mu)$\\
  \hline
  $A$ & ${k'}^{-2}[h-k^2(\xi(\xi+1)+\eta(\eta+1)+\mu(\mu+1)+\nu(\nu+1))]$ & $(\xi)(\eta)(\mu\nu)$ & $(\xi\eta)(\mu)(\nu)$ & $(\xi\mu\eta\nu)$ & $(\xi\nu\eta\mu)$\\
  \hline
  $B$ & $-h+\xi(\xi+1)+\eta(\eta+1)+\mu(\mu+1)+\nu(\nu+1)$ & $(\xi)(\mu)(\eta\nu)$ & $(\xi\eta\mu\nu)$
  & $(\xi\mu)(\eta)(\nu)$ & $(\xi\nu\mu\eta)$\\
  \hline
  $C$ & $hk^{-2}$ & $(\xi)(\nu)(\eta\mu)$ & $(\xi\eta\nu\mu)$ & $(\xi\mu\nu\eta)$ & $(\xi\nu)(\mu)(\eta)$\\
  \hline
  $D$ & $h{k'}^{-2}[-h+(\xi(\xi+1)+\eta(\eta+1)+\mu(\mu+1)+\nu(\nu+1))]$ & $(\xi)(\eta\mu\nu)$ & $(\xi\eta\nu)(\mu)$
  & $(\xi\mu\eta)(\nu)$ & $(\xi\nu\mu)(\eta)$\\
  \hline
  $E$ & $-hk^{-2}+\xi(\xi+1)+\eta(\eta+1)+\mu(\mu+1)+\nu(\nu+1)$ & $(\xi)(\eta\nu\mu)$ & $(\xi\eta\mu)(\nu)$ & $(\xi\mu\nu)(\eta)$ & $(\xi\nu\eta)(\mu)$\\
\end{tabular}
\end{center}
\end{table}}
\medskip

On the other hand, the accessory parameters of the Darboux equation in the Weierstrass form after transformation are listed in Table \ref{TT1}.
In this case, there are only three types of accessory parameters after the symmetry group transformations. Moreover, they are only constant multiples of the ``original" accessory parameter. It is because the potential of the Darboux equation in the Weierstrass form has an extra constant multiple after the actions of $G_{\mathrm{II}}$ (see (\ref{E:newdarbouxW}) in \S\ref{G2W}).
\medskip

\begin{table}
\caption{}\label{TT1}
\begin{center}

\begin{tabular}{c||c}
  ${X}$ & $h_X$\\
  \hline
  $I,\, C$ & $h$\\
  \hline
  $A,\, E$ & $(\tau-1)^2h$\\
  \hline
  $B,\, D$ & $\tau^2h$\\
\end{tabular}
\end{center}
\end{table}



\section{Expansion of local solutions: Darboux functions}\label{S:expansion}

Historically, it was Ince \cite[(1940)]{Ince1} who first considered the existence of \textit{Lam\'e functions} by developing infinite series expansions composed of Jacobian elliptic functions as solutions of the Lam\'e equation. Ince also had an alternative \textit{Fourier-Jacobi expansions} approach to Lam\'e equations \cite{Ince2}. This puts the study of Lam\'e functions in an equal standing as that of the Lam\'e polynomials.
Ince's work turns out to be important in Novikov's formulation of finite-gap potential of Lax pair representation of KdV with periodic boundary condition (see \cite{Dub,Novikov,IM,GW1} for instance). All these expansions are valid on certain domains lying on tori. We extend Ince's ideas to represent local solutions of the Darboux equation of Jacobian form in terms of infinite expansions of Jacobi elliptic functions.

In this section, let us define one local solution at $u=0$ with exponent $\xi+1$
and call it the {\it local Darboux solution}, denoted by $Dl(\xi,\eta,\mu,\nu;h;u,k)$. The expansions of $Dl$ at the other regular singular points $K,\, iK^\prime,\, K+iK^\prime$ can be obtained after applying the symmetries of the Darboux equation considered in \S\ref{S:g1} (for $G_\mathrm{I}$) and \S\ref{tran2} (for $G_\mathrm{II}$).
If $\xi=-\frac{3}{2},-\frac{5}{2},\cdots$, then $Dl(\xi,\eta,\mu,\nu;h;u,k)$ will generically be logarithmic and we do not discuss this degenerate case  further in this paper.

\begin{definition}
Suppose that $\xi\neq-\frac{3}{2},-\frac{5}{2},\cdots$. Let $Dl(\xi,\eta,\mu,\nu;h;u,k)$ be defined by the following series expansion

\begin{equation}\label{E:series}\sn(u,k)^{\xi+1}\cn(u,k)^{\eta+1}\dn(u,k)^{\mu+1}\sum_{m=0}^\infty C_m\sn(u,k)^{2m},
\end{equation}

where the coefficients $C_m(\xi,\eta,\mu,\nu;h;k)$ satisfy the relation
{\beqn\label{3term}&&(2m+2)(2m+2\xi+3)C_{m+1}\nonumber\\
&&\hspace{.5cm} +\{h-[2m+\eta+\xi+2]^2-k^2[2m+\mu+\xi+2]^2+(k^2+1)(\xi+1)^2\}C_{m}\nonumber\\
&&\hspace{1cm} +k^2(2m+\xi+\eta+\mu+\nu+2)(2m+\xi+\eta+\mu-\nu+1)C_{m-1}=0,\nonumber\\ \eeqn}
where $m\geq0$ and the initial conditions $C_{-1}=0$, $C_0=1$.
\end{definition}

\begin{remark}\label{R:why-jacobian} The factor in front of the summation sign of (\ref{E:series}) is multi-valued in general. Indeed if we expand this solution in terms of the Weierstrass elliptic form, then one would encounter one more ambiguity of having an extra square-root sign in addition to the complex power that appears in the first factor in front of the expansion \eqref{E:series}.
\end{remark}

For brevity, let
$$M_m(\xi):=(2m+2)(2m+2\xi+3);$$
$$L_m(\xi,\,\eta,\mu;h;k):=h-[2m+\eta+\xi+2]^2-k^2[2m+\mu+\xi+2]^2+(k^2+1)(\xi+1)^2;$$
$$K_m(\xi,\eta,\mu,\nu;k):=k^2(2m+\xi+\eta+\mu+\nu+2)(2m+\xi+\eta+\mu-\nu+1)$$
such that (\ref{3term}) can be written as
	\begin{equation}\label{E:recursion}
		M_mC_{m+1}+L_mC_m+K_mC_{m-1}=0,\ m\geq0.
	\end{equation}

We show that the series always converges for $|\sn(u,k)|<\min(1,|k|^{-1}).$ Before that, we look at the conditions for the series to be terminating,
which is the case we need not concern about the convergence.

\begin{theorem}\label{terminate}
If there exists a positive integer $q$ such that either 
\begin{equation}\label{E:termination}
		\xi+\eta+\mu+\nu=-2q-4\quad\mbox{ or }\quad \xi+\eta+\mu-\nu=-2q-3,
	\end{equation} 
holds,
then there exist $q+1$ values $h_0,\, \cdots,\, h_q$ of $h$ such that
the series $Dl(\xi,\eta,\mu,\nu;h_j;u,k)$ ($j=0,1,\cdots,q$) terminates.
\end{theorem}

\begin{proof}
It follows from $(\ref{3term})$ that the local Darboux solution becomes the finite series
$$\sn(u,k)^{\xi+1}\cn(u,k)^{\eta+1}\dn(u,k)^{\mu+1}\sum_{m=0}^q C_m\sn(u,k)^{2m} \mbox{ with } C_q\neq0$$
if and only if
$$K_{q+1}(\xi,\eta,\mu,\nu;k)=0=C_{q+1}(\xi,\eta,\mu,\nu;h;k).$$
Notice that $C_{q+1}(\xi,\eta,\mu,\nu;h;k)=0$ is equivalent to saying the vanishing of the finite continued-fraction
\beqn\label{fcf} f(\xi,\eta,\mu,\nu;h;k):=L_0/M_0-\frac{K_1/M_1}{L_1/M_1-}\frac{K_2/M_2}{L_2/M_2-}\cdots\frac{K_q/M_q}{L_q/M_q}=0.\eeqn

If $\xi,\eta,\mu,\nu$ are chosen such that $K_{q+1}(\xi,\eta,\mu,\nu;k)=0,$ that is,
(\ref{E:termination}) holds, then
there exist $q+1$ values $h_0,\, \cdots,\, h_q$ of $h$ such that (\ref{fcf}) holds,
and hence the series terminates.
\end{proof}
In this terminating case, we call the solution the {\it Darboux polynomial}, denoted by $Dp(\xi,\eta,\mu,\nu;h_j;u,k)$ ($j=0,\cdots,q$). {Note that the Darboux polynomials are generalisation of the classical Lam\'e polynomials which are eigensolutions to the Lam\'e equation. The Lam\'e polynomials are orthogonal polynomials which have some remarkable properties. Some of them are discovered only recently
(see e.g. Grosset and Veselov \cite{GV,GV1} (2006, 2008), Borcea and Shapiro \cite{BS} (2008)).}

\begin{remark}
The condition on $\xi$, $\eta$, $\mu$, $\nu$ so that a solution of Darboux equation exists in the form of a Darboux polynomial is indeed the condition that the Picard-Vessiot extension  of the given Darboux equation
is a Liouvillian extension of the field of elliptic functions. This criterion can also be obtained by a standard application of Kovacic algorithm \cite{DL}.
\end{remark}

Now we discuss the convergence when the series is in general non-terminating. 
\begin{theorem}\label{convergent}
Suppose that $Dl(\xi,\eta,\mu,\nu;h;u,k)$ is non-terminating, i.e., not a Darboux polynomial. Then it converges on the domain $\{|\sn u|<
\max(1,|k|^{-1})\}$ ($|k|\neq1$) if \beqn\label{ifcf} g(\xi,\eta,\mu,\nu;h;k):=L_0/M_0-\frac{K_1/M_1}{L_1/M_1-}\frac{K_2/M_2}{L_2/M_2-}\cdots=0\eeqn
holds. Otherwise, it converges only on the domain $\{|\sn u|<\min(1,|k|^{-1})\}$.
\end{theorem}

\begin{proof}

As
$$\lim_{m\ra\infty}M_m/4m^2=1,\ \lim_{m\ra\infty}L_m/4m^2=-1-k^2
\mbox{ and }
\lim_{m\ra\infty}K_m/4m^2=k^2,$$
it follows from Theorem \ref{poin}
that $\lim_{m\ra\infty}C_{m+1}/C_m$ exists, where the coefficient $C_m$ is defined in \eqref{E:recursion}, and is equal to
one of the roots of the quadratic equation
$$t^2-(1+{k}^2)t+k^2=0,$$
that is,
$\lim_{m\ra\infty}C_{m+1}/C_m=1\mbox{ or }k^2.$
In fact, the limit depends on whether the infinite continued fraction 
$g(\xi,\eta,\mu,\nu;h;k)=0$
holds or not.
By Theorem \ref{perron},
$$
\lim_{m\ra\infty}|C_{m+1}/C_m|=
\begin{cases}
\min(1,|k|^2)&\mbox{ if (\ref{ifcf}) holds}\\
\max(1,|k|^2)&\mbox{ otherwise}
\end{cases}.
$$
Then by the ratio-test,
the series converges for
$$|\sn u|<\begin{cases}
\max(1,|k|^{-1}) &\mbox{ if (\ref{ifcf}) holds}\\
\min(1,|k|^{-1}) &\mbox{ otherwise}
\end{cases}.$$
\end{proof}

\begin{definition}
If $h=\hat{h}$ is chosen such that (\ref{ifcf}) holds, then $Dl(\xi,\eta,\mu,\nu;h;u,k)$ converges on the larger domain $\{|\sn u|<
\max(1,|k|^{-1})\}$ and in this case we call the solution the {\it Darboux function}, denoted by $Df(\xi,\eta,\mu,\nu;\hat{h};u,k)$.
\end{definition}

\begin{remark}
When the parameters $\xi$, $\eta$, $\mu$ in the $Dl$ become zero (or $-1$), we recover Ince's series expansions
(see formulas (2.1), (3.1), (4.1), (4.3), (5.1), (5.3), (6.1), (6.3) in \cite{Ince1}).
\end{remark}

\section{The 192 solutions of the Darboux equation}\label{S:192}
In general, it is well-known that if a Fuchsian equation on $\mathbb{CP}^1$ has one of the exponent differences at a singular point is an integer, then one of the local  solutions corresponding to the singular point is generically logarithmic.
For the Darboux equation, it has a logarithmic solution only if
at least one of $\xi,\eta,\mu,\nu$ is in $\frac{2\Z+1}{2}$. One such exceptional case for $Dl(\xi,\eta,\mu,\nu;h;u,k)$ was already mentioned in the last section.
We shall not consider these degenerate cases henceforth.

Now using the automorphisms in $G\cong G_\mathrm{I}\rtimes_\Gamma G_\mathrm{II}$, we can generate $2\times2\times2\times24=192$ solutions of Darboux equation (\ref{E:darboux}) in the
following form when $\xi,\eta,\mu,\nu\notin\frac{2\Z+1}{2}$:
	\[
		 Dl\big(\sigma_{X_i}(\xi)^{s_\xi},\sigma_{X_i}(\eta)^{s_\eta},\sigma_{X_i}(\mu)^{s_\mu},\sigma_{X_i}(\nu);h_X;\tau_{X_i}(u,k),\kappa_X(k)\big),
	\]
where ${s_\xi},{s_\eta},{s_\mu}=+$ or $-$; $X=I,\, A,\,B,\, C,\, D,\, E$; $i=0,1,2,3$.
Note that the transformation $\sigma_{X_i}(\nu)\mapsto\sigma_{X_i}(\nu)^-$ does not give a new solution since by definition,
\beq &&Dl(\sigma_{X_i}(\xi)^{s_\xi},\sigma_{X_i}(\eta)^{s_\eta},\sigma_{X_i}(\mu)^{s_\eta},\sigma_{X_i}(\nu);h_X;\tau_{X_i}(u,k),\kappa_X(k))\\&=&
Dl(\sigma_{X_i}(\xi)^{s_\xi},\sigma_{X_i}(\eta)^{s_\eta},\sigma_{X_i}(\mu)^{s_\eta},\sigma_{X_i}(\nu)^-;h_X;\tau_{X_i}(u,k),\kappa_X(k)).\eeq

This is because the parameter $\nu$ does not present in the expansion of $Dl$.
Therefore each of 192 solutions can be identified by an element of
the Coxeter group $D_4\cong(\mathbb{Z}_2)^3\rtimes_\Gamma G_\mathrm{II}$.

Moreover, the 192 solutions split into 8 sets of 24 formally distinct but equivalent expressions, where each set defining the local
solution corresponding to one of the two exponents in the neighborhood of one of the four singular points. For each $s,\,t=+$ or $-$ and  $X=I,\,A,\,B,\, C,\, D,\, E$:
\smallskip

\eb
\item{the local solutions at $u=0$ with exponent $\xi+1$:}
$$Dl\big(\xi,\sigma_{X_0}(\eta)^s,\sigma_{X_0}(\mu)^t,\sigma_{X_0}(\nu);h_X;\tau_{X_0}(u,k),\kappa_X(k)\big),$$
\smallskip

\item{the local solutions at $u=0$ with exponent $-\xi$:}
$$Dl\big(-\xi-1,\sigma_{X_0}(\eta)^s,\sigma_{X_0}(\mu)^t,\sigma_{X_0}(\nu);h_X;\tau_{X_0}(u,k),\kappa_X(k)\big),$$
\smallskip

\item{the local solutions at $u=K(k)$ with exponent $\eta+1$:}
$$Dl\big(\eta,\sigma_{X_1}(\eta)^s,\sigma_{X_1}(\mu)^t,\sigma_{X_1}(\nu);h_X;\tau_{X_1}(u,k),\kappa_X(k)\big),$$
\smallskip

\item{the local solutions at $u=K(k)$ with exponent $-\eta$:}
$$Dl\big(-\eta-1,\sigma_{X_1}(\eta)^s,\sigma_{X_1}(\mu)^t,\sigma_{X_1}(\nu);h_X;\tau_{X_1}(u,k),\kappa_X(k)\big),$$
\smallskip

\item{the local solutions at $u=K(k)+iK'(k)$ with exponent $\mu+1$:}
$$Dl\big(\mu,\sigma_{X_2}(\eta)^s,\sigma_{X_2}(\mu)^t,\sigma_{X_2}(\nu);h_X;\tau_{X_2}(u,k),\kappa_X(k)\big),$$
\smallskip

\item{the local solutions at $u=K(k)+iK'(k)$ with exponent $-\mu$:}
$$Dl\big(-\mu-1,\sigma_{X_2}(\eta)^s,\sigma_{X_2}(\mu)^t,\sigma_{X_2}(\nu);h_X;\tau_{X_2}(u,k),\kappa_X(k)\big),$$
\smallskip

\item{the local solutions at $u=iK'(k)$ with exponent $\nu+1$:}
$$Dl\big(\nu,\sigma_{X_3}(\eta)^s,\sigma_{X_3}(\mu)^t,\sigma_{X_3}(\nu);h_X;\tau_{X_3}(u,k),\kappa_X(k)\big),$$
\smallskip

\item{the local solutions at $u=iK'(k)$ with exponent $-\nu$:}
$$Dl\big(-\nu-1,\sigma_{X_3}(\eta)^s,\sigma_{X_3}(\mu)^t,\sigma_{X_3}(\nu);h_X;\tau_{X_3}(u,k),\kappa_X(k)\big).$$
\ee
The above list gives all the 192 local series expansions for the Darboux equation.

\bigskip

The terminating conditions for the above 192 series can easily be obtained by applying the symmetry group $G\cong G_\mathrm{I}\rtimes_\Gamma G_\mathrm{II}$ to that for the $Dl$ in (\ref{E:termination}). We record this result in the following theorem.
\medskip

\begin{theorem}
If there exists a positive integer $q$ such that either 
$$\sigma_{X_i}(\xi)^{s_\xi}+\sigma_{X_i}(\eta)^{s_\eta}+\sigma_{X_i}(\mu)^{s_\mu}+\sigma_{X_i}(\nu)=-2q-4$$
or $$\sigma_{X_i}(\xi)^{s_\xi}+\sigma_{X_i}(\eta)^{s_\eta}+\sigma_{X_i}(\mu)^{s_\mu}-\sigma_{X_i}(\nu)=-2q-3,$$ 
holds, and
there exists a value of $h$ such that $h_X(\xi,\eta,\mu,\nu;h;k)$ satisfies the finite continued fraction
$$f(\sigma_{X_i}(\xi)^{s_\xi},\sigma_{X_i}(\eta)^{s_\eta},\sigma_{X_i}(\mu)^{s_\mu},\sigma_{X_i}(\nu);h_X;\kappa_X(k))=0,$$ 
then the series $Dl(\sigma_{X_i}(\xi)^{s_\xi},\sigma_{X_i}(\eta)^{s_\eta},\sigma_{X_i}(\mu)^{s_\mu},\sigma_{X_i}(\nu);h_X;\tau_{X_i}(u,k),\kappa_X(k))$
terminates.
\end{theorem}
\medskip

We record the remaining case when the infinite continued-fraction in (\ref{ifcf}) vanishes as
\medskip

\begin{theorem} If $h$ is chosen such that $h_X(\xi,\eta,\mu,\nu;h;k)$ satisfies the infinite continued fraction
$$g(\sigma_{X_i}(\xi)^{s_\xi},\sigma_{X_i}(\eta)^{s_\eta},\sigma_{X_i}(\mu)^{s_\mu},\sigma_{X_i}(\nu);h_X;\kappa_X(k))=0,$$ 
then the series $Dl(\sigma_{X_i}(\xi)^{s_\xi},\sigma_{X_i}(\eta)^{s_\eta},\sigma_{X_i}(\mu)^{s_\mu},\sigma_{X_i}(\nu);h_X;\tau_{X_i}(u,k),\kappa_X(k))$
converges on the larger domain $\{|\sn u|<\max(1,|\kappa_X(k)|^{-1})\}$ for $|\kappa_X(k)|^{-1}\neq1$.
\end{theorem}
\bigskip

\section{Special cases}\label{S:Special}

Here we consider the symmetries of a number of special cases of the Darboux equation we have found. These special cases of differential equations have been studied
 by mathematicians since the beginning of the 19th century. We first return to the Lam\'e equation which we encountered in the Introduction.

\subsection{Lam\'e equation: (Eqn(\ref{E:darboux}) for $\xi,\eta,\mu=0$ or $-1$)} \label{LameEqn}

The \textit{Lam\'e equation}
	\beqn \label{Lame}
		\frac{{d}^{2}y}{{du}^{2}}+\big(h-\nu(\nu+1)k^2{\ts\sn^2u}\big)y=0,
	\eeqn
{with the Riemann $P$-scheme
\[P_{\mathbb{C}\slash\Lambda}\begin{Bmatrix}
\ iK'(k)\ &\\
\nu+1&u;\, h\\
-\nu&
\end{Bmatrix},\]}
\noindent was first derived by Lam\'e in 1837 \cite[XXIII]{WW} after separation of variables of Laplace's equation in three-space under the \textit{ellipsoidal coordinates}. Since then, the equation has been worked on by well-known mathematicians such as Heine \cite[(1878)]{Heine}, Hermite \cite[(1878)]{Hermite2}, Hobson \cite[(1892)]{Hobson1} in the 19th century, and Ince \cite[(1940)]{Ince1,Ince2} and Erd\'elyi \cite[(1940)]{Erdelyi3} in the 20th century. The Lam\'e equation has recently appeared in Knot theory \cite{IS1999}, PDEs and algebraic geometry \cite{CLW}, Integrable systems \cite{GW1,GW2} and finite-gap theory \cite{Novikov} which is already mentioned.


{
Clearly, the Lam\'e equation remains unchanged under the transformation $\nu\mapsto\nu^\pm$,
which induces the subgroup of $G_\mathrm{I}$ isomorphic to $\Z_2.$

Moreover, the transformation $A$ sends the half-period $\omega_3=iK'$ to $\omega_2=K+iK'$. Hence the transformation $A_1$ sends $\omega_3=iK'$ back to $\omega_2+\omega_1=\omega_3$. This transformation $A_1$ fixed the singularity $\omega_3$ and hence the form of the Lam\'e equation. Similarly, one reads off from
Table \ref{T1} that there exist only six transformations $X_i=I_0, A_1, B_2,C_0,D_2,E_1$ such that the Lam\'e equation preserves the form, that is, under such six transformations $X_i$, the equation (\ref{Lame}) becomes

\beqn\label{Lame1}
\frac{{d}^{2}y}{{dw}^{2}}+\Big(h_X-\nu(\nu+1)\kappa_X^2{\ts\sn^2(w,\kappa_X)}\Big)y=0,\eeqn
where $w=\tau_{X_i}(u,k)$, $\kappa_X$ and $h_X$ are given in Table \ref{T1} and \ref{T2} respectively.
These six transformations induce a subgroup of $G_\mathrm{II}$ isomorphic to $S_3.$
}


\subsection{Associated Lam\'e equation: (Eqn(\ref{E:darboux}) for $\xi,\eta=0$ or $-1$)}

{The associated Lam\'e equation

\beqn\label{AL}
\frac{{d}^{2}y}{{du}^{2}}+\Big(h-\mu(\mu+1)k^2{\ts\cd^2(u,k)}-\nu(\nu+1)k^2{\ts\sn^2(u,k)}\Big)y=0
\eeqn

has the Riemann $P$-scheme
\[P_{\mathbb{C}\slash\Lambda}\begin{Bmatrix}
\ K(k)+iK'(k)\ &\ iK'(k)\ &\\
\mu+1&\nu+1&u;\, h\\
-\mu&-\nu&
\end{Bmatrix}.\]

Clearly, the associated Lam\'e equation remains unchanged under $\mu\mapsto\mu^\pm$ and $\nu\mapsto\nu^\pm$,
which induce the subgroup of $G_\mathrm{I}$ isomorphic to $\Z_2\times\Z_2.$

Moreover, it follows from
Table \ref{T1} that there exist only four transformations $X_i=I_0, I_1, A_0,A_1$ such that the  associated Lam\'e equation
preserves the form, that is, under such four transformations $X_i$, the equation (\ref{AL}) becomes
{\begin{eqnarray}\label{AL1}
&&\frac{{d}^{2}y}{{dw}^{2}}+\Big(h_X-\sigma_{X_i}(\mu)(\sigma_{X_i}(\mu)+1)\kappa_X^2
{\ts\cd^2(u,\kappa_X)}\nonumber\\
&&\hspace{2cm}-\sigma_{X_i}(\nu)(\sigma_{X_i}(\nu)+1)\kappa_X^2{\ts\sn^2(w,\kappa_X)}\Big)y=0,\end{eqnarray}}
where $w=\tau_{X_i}(u,k)$, $\kappa_X$ and $h_X,\sigma_{X_i}$ are given in Table \ref{T0} and \ref{T2} respectively.
These four transformations induce a subgroup of $G_\mathrm{II}$ isomorphic to $\Z_2\times\Z_2$.
}


\section{Other transformations of the Darboux equation}
The exponent parameters $\xi,\eta,\mu,\nu$ are unrestricted in the transformation in $G_\mathrm{II}$.
In the section, we study the other kinds of the transformations of Darboux equation with one of the following restrictions: either ($\xi=\eta$, $\mu=\nu$) or
($\xi=\eta=\mu=\nu$).

For the case $\xi=\eta$, $\mu=\nu$, we consider the Landen transformation $L:(u,k)\mapsto \big((1+k')u,\frac{1-k'}{1+k'}\big)$
and under the transformation, we obtain the following \cite[p. 372]{EAM2}
$$\sn\bigg((1+k')u,\frac{1-k'}{1+k'}\bigg)=\frac{(1+k')\sn(u,k)\cn(u,k)}{\dn(u,k)};$$
$$\cn\bigg((1+k')u,\frac{1-k'}{1+k'}\bigg)=\frac{1-(1+k')\sn^2(u,k)}{\dn(u,k)};$$
$$\dn\bigg((1+k')u,\frac{1-k'}{1+k'}\bigg)=\frac{1-(1-k')\sn^2(u,k)}{\dn(u,k)}.$$
Then we can derive
$$(1+k')\dn((1+k')u,\frac{1-k'}{1+k'})=\dn(u,k)+k'\nd(u,k);$$
$$(1+k')\cs((1+k')u,\frac{1-k'}{1+k'})=\cs(u,k)-k'\sc(u,k).$$
Thus, let $w=(1+k')u$ and $\kappa^*=\frac{1-k'}{1+k'}$, we have
$$\ts(1+k')^2\kappa^{*2}\sn^2(w,\kappa^*)+k^2=k^2\sn^2(u,k)+k^2\cd^2(u,k);$$
$$\ts(1+k')^2\kappa^{*2}\ns^2(w,\kappa^*)+k^2=\ns^2(u,k)+\dc^2(u,k).$$
So under the Landen transformation $L$, the Darboux equation

{\beq
&&\frac{{d}^{2}y}{{du}^{2}}+\Big(h-{\xi(\xi+1)}{\ts\ns^2(u,k)}-{\ts\xi(\xi+1)\dc^2(u,k)}\nonumber\\&&\ \hspace{2cm} -
{\ts\nu(\nu+1)k^2\cd^2(u,k)}-\nu(\nu+1)k^2{\ts\sn^2(u,k)}\Big)y=0
\eeq} becomes
\beq\frac{{d}^{2}y}{{dw}^{2}}+\Big(h^*-\xi(\xi+1)}{\ts\ns^2(w,\kappa^*)-\nu(\nu+1)\kappa^{*2}{\ts\sn^2(w,\kappa^*)}\Big)y=0,\eeq
where $h^*=[h-k^2\xi(\xi+1)-k^2\nu(\nu+1)]/(1+k')^2$. Therefore, we obtain

\begin{theorem}[Landen Transformation Formula]
For $\xi=-\frac{3}{2},-\frac{5}{2},\cdots$, we have
$$Dl\bigg(\xi,0,0,\nu;h^*;(1+k')u,\frac{1-k'}{1+k'}\bigg)=(1+k')^{\xi+1}Dl(\xi,\xi,\nu,\nu;h;u,k),$$
where $h^*=[h-k^2\xi(\xi+1)-k^2\nu(\nu+1)]/(1+k')^2$.
\end{theorem}

For the case $\xi=\eta=\mu=\nu$, it follows from the duplication formula for Jacobi elliptic function
$${\ts\sn^2(u,k)}=\frac{1-\cn(2u,k)}{1+\dn(2u,k)}$$
that
$${\ts\ns^2(u,k)}+{\ts\dc^2(u,k)}+{\ts k^2\cd^2(u,k)}+k^2{\ts\sn^2(u,k)}=4{\ts\ns^2(2u,k)}.$$
So the Darboux equation
{\beq
&&\frac{{d}^{2}y}{{du}^{2}}+\Big(h-{\xi(\xi+1)}{\ts\ns^2(u,k)}-{\ts\xi(\xi+1)\dc^2(u,k)}\nonumber\\&&\ \hspace{2cm} -
{\ts\xi(\xi+1)k^2\cd^2(u,k)}-\xi(\xi+1)k^2{\ts\sn^2(u,k)}\Big)y=0
\eeq} becomes
\beqn\label{LameOther}\frac{{d}^{2}y}{{d\tilde u}^{2}}+\Big(h/4-\xi(\xi+1){\ts\ns^2(\tilde u,k)}\Big)y=0,\eeqn
where $\tilde u=2u.$

\begin{remark} Note that this form (\ref{LameOther}) of Lam\'e equation has been discussed in
Whittaker-Watson \cite[p. 555]{WW}. \end{remark}

Thus we obtain

\begin{theorem}[Duplication Formula] 
For $\xi=-\frac{3}{2},-\frac{5}{2},\cdots$, we have
$$Dl(\xi,0,0,0;h/4;2u,k)=2^{\xi+1}Dl(\xi,\xi,\xi,\xi;h;u,k).$$
\end{theorem}

\appendix

\section{Results on three-term recursion relations}\label{S:three term}

In this section, we review some useful results about three-term recursion relations. We refer to the readers to Gautschi \cite{Gautschi} for more details.

Given the three-term recursion
\[R_rC_{r+1}+S_rC_r+P_rC_{r-1}=0\, (r=0,1,2,\cdots),\]
where $C_{-1}=0$ and $R_r\neq0$ for all $r=0,1,2,\cdots$.
Assume that $\lim_{r\to\infty} P_r:=P,$  $\lim_{r\to\infty} S_r:=S$ and $\lim_{r\to\infty} R_r:=R$ exist.
The limit of ${C_{r+1}}/{C_r}$ ($r\to\infty$) can be determined by Poincar\'e's Theorem and Perron's Theorem.

\begin{theorem}[Poincar\'e's Theorem (see \cite{Poincare} or {\cite[p.527]{MT}})]\label{poin}
The limit
\[\lim_{r\to\infty} \frac{C_{r+1}}{C_r}= t_1\ \mbox{ or }\ t_2,\]
where $t_1$ and $t_2$ are the roots of the quadratic equation $Rt^2+St+P$.
\end{theorem}

\begin{theorem}[Perron's Theorem (see {\cite[\S 57]{Perron}})]\label{perron}
Suppose that $|t_1|<|t_2|.$ If
the infinite continued fraction
\[S_0/R_0-\frac{P_1/R_1}{S_1/R_1-}\frac{P_2/R_2}{S_2/R_2-}\cdots=0\] holds, then
\[\lim_{r\to\infty} \Bigg|\frac{C_{r+1}}{C_r}\Bigg|= t_1.\]
Otherwise, \[\lim_{r\to\infty} \Bigg|\frac{C_{r+1}}{C_r}\Bigg|= t_2.\]
\end{theorem}

\section{Asymmetric form of Sparre equation}\label{S:Sparre}

Sparre (\cite{Sparre}) in 1883 discovered the following related equation
\begin{eqnarray}
\label{E:sparre}
&&\frac{{d}^{2}\tilde{y}}{{du}^{2}}+
\Big[2\ell\frac{k^2\sn(u,k)\cn(u,k)}{\dn(u,k)}+2\ell_1\frac{\sn(u,k)\dn(u,k)}{\cn(u,k)}-2\ell_2\frac{\cn(u,k)\dn(u,k)}{\sn(u,k)}\Big]\frac{{d}\tilde{y}}{{du}}
\nonumber\\&&\ +
\Big[\tilde{h}-{(\xi-\ell_2)(\xi+\ell_2+1)}\,{\ts\ns^2(u,k)}-{\ts(\eta-\ell_1)(\eta+\ell_1+1)\dc^2(u,k)}\nonumber\\&&\
{-\ts(\mu-\ell)(\mu+\ell+1)k^2\cd^2(u,k)}-(\nu+\ell+\ell_1+\ell_2)(\nu-\ell-\ell_1-\ell_2+1)k^2{\ts\sn^2(u,k)}\Big]\,\tilde{y}=0,\nonumber\\
\end{eqnarray} where the Riemann $P$-scheme
\[P\begin{Bmatrix}
\ 0\ &\ K(k)\ &\ K(k)+iK'(k)\ &\ iK'(k)\ &\\
\xi+1+\ell_2&\eta+1+\ell_1&\mu+1+\ell&\nu+1-\ell-\ell_1-\ell_2&u;\, \tilde{h}\\
-\xi+\ell_2&-\eta+\ell_1&-\mu+\ell&-\nu-\ell-\ell_1-\ell_2&
\end{Bmatrix}.\]
The Sparre equation (\ref{E:sparre}) can be normalized into the following two versions:
\eb
\item (Symmetric form) the Darboux equation (\ref{E:darboux}) by using the transformation $\tilde{y}=(\sn u)^{\ell_2}(\cn u)^{\ell_1}(\dn u)^{\ell}y$ (see \cite[\S 3]{MS});
\item (Asymmetric form)
\ee
\begin{eqnarray}
\label{E:asym}
&&\frac{{d}^{2}y}{{du}^{2}}+
\Big[2\mu\frac{k^2\sn(u,k)\cn(u,k)}{\dn(u,k)}+2\eta\frac{\sn(u,k)\dn(u,k)}{\cn(u,k)}-2\xi\frac{\cn(u,k)\dn(u,k)}{\sn(u,k)}\Big]\frac{{d}y}{{du}}
\nonumber\\&&\ +
\Big[h-(\nu+\xi+\eta+\mu)(\nu-\xi-\eta-\mu+1)k^2{\ts\sn^2(u,k)}\Big]\,y=0,\nonumber\\
\end{eqnarray} where the Riemann $P$-scheme
\[P\begin{Bmatrix}
\ 0\ &\ K(k)\ &\ K(k)+iK'(k)\ &\ iK'(k)\ &\\
0&0&0&\nu-\xi-\eta-\mu+1&u;\, h\\
2\xi+1&2\eta+1&2\mu+1&-\nu-\xi-\eta-\mu&
\end{Bmatrix},\]
by using the transformation $\tilde{y}=(\sn u)^{\ell_2-\xi}(\cn u)^{\ell_1-\eta}(\dn u)^{\ell-\mu}y$.

The automorphism group of the Sparre equation in asymmetric form is isomorphic to the Coxeter group $D_4\cong(\mathbb{Z}_2)^3\rtimes S_4$ which
is induced by the tranformations $w=\tau_{X_i}(u,k)$ (change of the independent variable from $u$ to $w$) by the map $X_i$ in Table \ref{T1}, followed by
the tranformations $y(w)=(\sn w)^{n_1}(\cn w)^{n_2}(\dn w)^{n_3}z(w)$ (change of the dependent variable from $y(w)$ to $z(w)$) for two suitable
values of $n_1, n_2, n_3$ such that the Sparre equation preserves the asymmetric form. (Here the transformation interchanging the exponents at $iK'(k)$ is excluded.)

\subsection{Special cases for the asymmetric form}
\subsubsection{Picard equation \cite{Picard}:(\textbf{Eqn}(\ref{E:asym}) for $\mu=-\nu=n/2,\ \eta=\xi=0,\ \alpha=h$)}
The Picard equation
\beqn\label{Picard}\frac{{d}^{2}y}{{du}^{2}}+ nk^2\frac{\ts\sn (u,k)\ts\cn (u,k)}{\ts\dn (u,k)}	 \frac{{d}y}{{du}} +\alpha y=0 \eeqn
has the Riemann $P$-scheme
\[
	P_{\mathbb{C}\slash\Lambda}\begin{Bmatrix}
K(k)+iK'(k)&iK'(k)&\\
0&0&u;\, \alpha\\
1+n&1-n&
\end{Bmatrix}.
\]
It follows from
Table \ref{T1} that there exist only four transformations $X_i=I_0, I_1, A_0,A_1$ such that the Picard equation
preserves the form, that is, under such four transformations $X_i$, the equation (\ref{Picard}) becomes
\beqn\label{Picard1}\frac{{d}^{2}y}{{dw}^{2}}+n_{X_i}\kappa_X^2\frac{\ts\sn(w,\kappa_X)\ts\cn (w,\kappa_X)}{\ts\dn (w,\kappa_X)}	
\frac{{d}y}{{dw}} +\alpha_X y=0,\eeqn
where $w=\tau_{X_i}(u,k)$, $\kappa_X$ are given in Table \ref{T1} and $n_{X_i}(n),\alpha_X(\alpha,k)$ are given in Table \ref{T3}.
These four transformations induce the automorphism group of the Picard equation isomorphic to $\Z_2\times\Z_2$.
\begin{remark}
The tranformations $y(w)=(\sn w)^{n_1}(\cn w)^{n_2}(\dn w)^{n_3}z(w)$ for any values $n_1, n_2, n_3$ do not preserve the form
of Picard equation.
\end{remark}

\begin{table}
\begin{center}
\caption{}\label{T3}
\begin{tabular}{c||c|c}
  ${X_i}$ & ${n_{X_i}(n)}$ & $\alpha_X(\alpha,k)$ \\
  \hline
  $I_0$ & $n$ & $\alpha$ \\
  $I_1$ & $-n$ & \\
  \hline
  $A_0$ & $-n$ & $\alpha/(k')^{2}$\\
  $A_1$ & $n$ & \\
  \end{tabular}
\end{center}
\end{table}

\subsubsection{Hermite equation \cite{Hermite1}:(\textbf{Eqn}(\ref{E:asym}) for $\mu=\lambda+1,\ \nu=n,\ \eta=\xi=0$)}
The Hermite equation
{\begin{eqnarray}\label{Hermite}
&&\frac{{d}^{2}y}{{du}^{2}}+ 2(\lambda+1)k^2\frac{\ts\sn (u,k)\ts\cn (u,k)}{\ts\dn (u,k)}\frac{{d}y}{{du}}\nonumber\\
&&\hspace{2cm} -\big(h+(n-\lambda)(n+\lambda+1)k^2{\ts\sn^2(u,k)}\big)y=0 \end{eqnarray}}
has the Riemann $P$-scheme
\[P_{\mathbb{C}\slash\Lambda}\begin{Bmatrix}
K(k)+iK'(k)&iK'(k)&\\
0&n-\lambda&u;\, h\\
2\lambda+3&-n-\lambda-1&
\end{Bmatrix}.\]

The automorphism group of the Hermite equation has three generators induced by the following transformations (preserving the form of the equation):
\eb
\item $w=\tau_{A_1}(u,k)$ (change of the independent variable from $u$ to $w$) by the map $A_1$ in Table \ref{T1};
\item $y(u)=(\dn u)^{2\lambda+3}z(u)$ (change of the dependent variable from $y(u)$ to $z(u)$);
\item $w=\tau_{A_0}(u,k)$ (change of the independent variable from $u$ to $w$) by the map~$A_0$ given in Table \ref{T1}, followed by
      $y(w)=(\dn w)^{-n-\lambda-1}z(w)$ (change of the dependent variable from $y(w)$ to $z(w)$).
\ee
Since each generator has order two, the automorphism group is isomorphic to~$(\Z_2)^3$.

\begin{remark}
Both Picard and Hermite equations can be normalized into the form of the associated Lam\'e equation (\ref{AL}) by a suitable transformation.
\end{remark}

\section{Elliptic functions}\label{S:conversion}
Let $\wp(z)=\wp(z|\omega,\omega')$ be the Weierstrass elliptic function of periods $2\omega,2\omega'$ with a double pole at $z=0,$ where
$\Im(\omega'/\omega)>0.$ Let
$$\omega_1=\omega,\ \omega_2=-\omega-\omega',\ \omega_3=\omega',$$
$$e_i=\wp(\omega_i),\ i=1,2,3.$$
Now for the Jacobi elliptic function,
the modulus $k$ and the complementary modulus $k'$ can be expressed in terms of $e_1,e_2,e_3$:
\beq\label{modulus} k=\frac{(e_2-e_3)^{1/2}}{(e_1-e_3)^{1/2}},\ {k'}=\frac{(e_1-e_2)^{1/2}}{(e_1-e_3)^{1/2}}.\eeq
Let \beq\label{variable} u=(e_1-e_3)^{1/2}z.\eeq Then
{\beq\label{jacobi}\sn(u,k)=\frac{(e_1-e_3)^{1/2}}{[\wp(z)-e_3]^{1/2}},\ \cn(u,k)=\frac{[\wp(z)-e_1]^{1/2}}{[\wp(z)-e_3]^{1/2}},\ \dn(u,k)=\frac{[\wp(z)-e_2]^{1/2}}{[\wp(z)-e_3]^{1/2}}.\eeq}
and the other nine Jacobi elliptic functions are defined to be the reciprocals of these three functions and the quotients of any two of them. Also,
the quarter periods
\beq \label{period} K(k)=(e_1-e_3)^{1/2}\omega,\ iK'(k)=(e_1-e_3)^{1/2}\omega'.\eeq

\section{Semi-direct products}\label{SD}

\begin{definition}\label{semidirect}
Given any two groups $N$ and $H$ 
and a group homomorphism $\Gamma : H \ra Aut(N)$, we can construct a new group $N\rtimes_{\Gamma}H$, called the (outer) semidirect
product of $N$ and $H$ with respect to $\Gamma$, defined as follows:
\eb
\item As a set, $N\rtimes_{\Gamma}H$ is the cartesian product $N \times H$;
\item Multiplication of elements in $N\rtimes_{\Gamma}H$ is determined by the homomorphism $\Gamma$. The operation is
$$*: (N\rtimes_{\Gamma} H)\times(N\rtimes_{\Gamma} H)\to N\rtimes_{\Gamma} H$$
defined by $$(n_1, h_1)*(n_2, h_2) = (n_1\Gamma_{h_1}(n_2), h_1h_2)$$
for $n_1, n_2 \in N$ and $h_1, h_2\in H$.
\ee
\end{definition}

The idea of semi-direct product can be expressed in the language of exact sequence also, which is the next definition.

\begin{definition}
Let $L$, $M$, $N$ be groups. A short exact sequence is the collection of group homomorphisms
$$
0\to L\to M\to N\to 0
$$
such that the image of each of these homomorphisms equals to the kernel of the next homomorphism.
\end{definition}

\begin{definition}
Two exact sequences
$$
\begin{array}{c}
0\to L\to M\to N\to 0\\
0\to L\to M'\to N\to 0
\end{array}
$$
are equivalent if there exists a homomorphism $M\to M'$ such that
$$
\begin{array}{ccccc}
0\to&L&\to M\to&N&\to 0\\
 &||&\downarrow&||& \\
0\to&L&\to M'\to&N&\to 0
\end{array}
$$
is a commutative diagram. It can be shown that equivalence of exact sequences is an equivalence relation.
\end{definition}

\begin{definition}
Given an exact sequence
$$
0\to L\to M\stackrel{\pi}{\to} N\to 0,
$$
a choice of splitting (or a section) is a homomorphism $N\stackrel{\iota}{\to}M$ such that
$\pi\circ\iota$ is the identity map of $N$.
\end{definition}

\begin{example}
Fix $n\in\mathbb{N}$ and let $R$ be the group of rigid motions in $\mathbb{R}^n$ (that is, the group
of continuous maps $\mathbb{R}^n\to\mathbb{R}^n$ which preserve the Euclidean norm). Then there is
an exact sequence
$$
0\to\mathbb{R}^n\to R\to O(n)\to 0.
$$
This is just another way of saying that $R$ is a semi-direct product of $\mathbb{R}^n$ (identified
as the group of translations) and $O(n)$. A choice of section is easy also, which is simply the
inclusion $O(n)\to R$.
\end{example}

The languages of exact sequence (with splitting) and semi-direct product are equivalent, in the sense of the following identification.

\begin{theorem}[see {\cite[p.65-68]{Bourbaki}}]
Given two groups $N$ and $H$. There is a one-to-one correspondence between the isomorphism classes of
semi-direct products $N\rtimes H$ and the equivalence classes of exact sequences
$0\to N\to...\to H\to 0$ with a section.
\end{theorem}


\begin{thebibliography}{100}

\bibitem{ahlfors}
Ahlfors, L.V.: Complex Analysis.
 \newblock McGraw-Hill Co. (1979)

\bibitem{AR}
Arscott, F.M., Reid, G.D.L.: Some linear transformations of {L}am\'e functions
  and ellipsoidal wave functions.
\newblock Bul. Inst. Politehn. Ia\c si (N.S.) \textbf{17(21)}(1-2, part 1),
  79--86 (1971)

\bibitem{BD}
Baldassarri, F., Dwork, B.: On second order linear differential equations with
  algebraic solutions.
\newblock Amer. J. Math. \textbf{101}(1), 42--76 (1979)

\bibitem{BS}
Borcea, J., Shapiro, B.: Root asymptotics of spectral polynomials for the
  {L}am\'e operator.
\newblock Comm. Math. Phys. \textbf{282}(2), 323--337 (2008)

\bibitem{Bourbaki}
Bourbaki, N.: Algebra {I}. {C}hapters 1--3.
\newblock Elements of Mathematics (Berlin). Springer-Verlag, Berlin (1998)

\bibitem{BF}
Byrd, P.F., Friedman, M.D.: Handbook of elliptic integrals for engineers and
  scientists.
\newblock Die Grundlehren der mathematischen Wissenschaften, Band 67.
  Springer-Verlag, New York-Heidelberg (1971).
\newblock Second edition, revised

\bibitem{CLW}
Chai, C.L., Lin, C.S., Wang, C.L.: Mean field equations, hyperelliptic curves
  and modular forms: {I}.
\newblock Camb. J. Math. \textbf{3}(1-2), 127--274 (2015)

\bibitem{CH}
Courant, R., Hilbert, D.: Methods of mathematical physics. {V}ol. {I}.
\newblock Interscience Publishers, Inc., New York, N.Y. (1953)

\bibitem{Darboux}
Darboux, G.: Sur une \'equation lin\'eaire.
\newblock T. XCIV, Comptes Rendus de l'Academie des Sciences pp. 1645--1648
  (1882)

\bibitem{Dub}
Dubrovin, B.A.: Periodic problems for the {K}orteweg-de {V}ries equation in the
  class of finite band potentials.
\newblock Functional Anal. Appl. \textbf{9}, 215--223 (1975)

\bibitem{DL}
Duval, A., Loday-Richaud, M.: Kova\v ci\v c's algorithm and its application to
  some families of special functions.
\newblock Appl. Algebra Engrg. Comm. Comput. \textbf{3}(3), 211--246 (1992)

\bibitem{Erdelyi3}
Erd{\'e}lyi, A.: On {L}am\'e functions.
\newblock Philos. Mag. (7) \textbf{31}, 123--130 (1941)

\bibitem{EAM2}
Erd{\'e}lyi, A., Magnus, W., Oberhettinger, F., Tricomi, F.G.: Higher
  {T}ranscendental {F}unctions. {V}ol. {II}.
\newblock Robert E. Krieger Publishing Co. Inc., Melbourne, Fla. (1981)

\bibitem{Gautschi}
Gautschi, W.: Computational aspects of three-term recurrence relations.
\newblock SIAM Rev. \textbf{9}, 24--82 (1967)

\bibitem{GW1}
Gesztesy, F., Weikard, R.: Lam\'e potentials and the stationary (m){K}d{V}
  hierarchy.
\newblock Math. Nachr. \textbf{176}, 73--91 (1995)

\bibitem{GW2}
Gesztesy, F., Weikard, R.: Picard potentials and {H}ill's equation on a torus.
\newblock Acta Math. \textbf{176}(1), 73--107 (1996)

\bibitem{GV}
Grosset, M.P., Veselov, A.P.: Elliptic {F}aulhaber polynomials and {L}am\'e
  densities of states.
\newblock Int. Math. Res. Not. pp. Art. ID 62,120, 31 (2006)

\bibitem{GV1}
Grosset, M.P., Veselov, A.P.: Lam\'e equation, quantum {E}uler top and elliptic
  {B}ernoulli polynomials.
\newblock Proc. Edinb. Math. Soc. (2) \textbf{51}(3), 635--650 (2008)

\bibitem{HM}
Harris, J., Morrison, I.: Moduli of Curves.
\newblock Springer Verlag, GTM (1991)

\bibitem{Heine}
Heine, E.: Auszug eines {S}chreibens uber die {L}ame\'schen {F}unctionen an den
  {H}erausgeber.
\newblock J. reine angew. Math. \textbf{56}, 79--86 (1859)

\bibitem{Hermite2}
Hermite, C.: Sur l'\'equation de lam\'e.
\newblock Annali di Matematica Pura ed Applicata \textbf{9}, 21--24 (1878)

\bibitem{Hermite1}
Hermite, C.: Sur l'int\'egration de l'\'equation diff\'erentielles de lam\'e.
\newblock Journal f\"ur die reine und angewandte Mathematik \textbf{89}, 9--18
  (1880)

\bibitem{Heun}
Heun, K.: Zur {T}heorie der {R}iemann'schen {F}unctionen zweiter {O}rdnung mit
  vier {V}erzweigungspunkten.
\newblock Math. Ann. \textbf{33}(2), 161--179 (1888)

\bibitem{Hobson1}
Hobson, E.W.: The harmonic functions for the elliptic cone.
\newblock Proc. London Math Soc. \textbf{23}, 231--240 (1892)

\bibitem{Ince2}
Ince, E.L.: Further investigations into the periodic {L}am\'e functions.
\newblock Proc. Roy. Soc. Edinburgh \textbf{60}, 83--99 (1940)

\bibitem{Ince1}
Ince, E.L.: The periodic {L}am\'e functions.
\newblock Proc. Roy. Soc. Edinburgh \textbf{60}, 47--63 (1940)

\bibitem{IM}
Its, A.R., Matveev, V.B.: Schr\"odinger operators with the finite-gap spectrum
  and the {$N$}-soliton solutions of the {K}orteweg-de {V}ries equation.
\newblock Theoret. and Math. Phys. \textbf{23}, 343--355 (1975)

\bibitem{IS1999}
Ivey, T.A., Singer, D.: Knot types, homotopies and stability of closed elastic
  rods.
\newblock Proc. London Math. Soc. (3), 429--450  (1999)

\bibitem{Katz}
Katz, N.M.: Algebraic solutions of differential equations ({$p$}-curvature and
  the {H}odge filtration).
\newblock Invent. Math. \textbf{18}, 1--118 (1972)

\bibitem{Lame}
Lam\'e, G.: Sur les surfaces isothermes dans les corps solides homog\`enes en
  \'equilibre de temperature.
\newblock J. Math. Pures Appl. \textbf{2}, 147--188 (1837)

\bibitem{Lehner1964} Lehner, J.,: Discontinuous Groups and Automorphic Functions. 
\newblock Math. Surveys \& Mono., no. 8, Amer. Math. Soc. 1964.
\bibitem{Maier1}
Maier, R.S.: The 192 solutions of the {H}eun equation.
\newblock Math. Comp. \textbf{76}(258), 811--843 (2007)

\bibitem{Maier}
Maier, R.S.: {$P$}-symbols, {H}eun identities, and {${}_3F_2$} identities.
\newblock Contemp. Math. \textbf{471}, 139--159 (2008)

\bibitem{MS}
Matveev, V.B., Smirnov, A.O.: On the link between the {S}parre equation and
  {D}arboux-{T}reibich-{V}erdier equation.
\newblock Lett. Math. Phys. \textbf{76}(2-3), 283--295 (2006)

\bibitem{MT}
Milne-Thomson, L.M.: The {C}alculus of {F}inite {D}ifferences.
\newblock Macmillan and Co., Ltd., London (1951)

\bibitem{Novikov}
Novikov, S.P.: The periodic problem for the {K}orteweg-de {V}ries equation.
\newblock Functional Anal. Appl. \textbf{8}, 236--246 (1974)

\bibitem{Perron}
Perron, O.: Die {L}ehre von den {K}ettenbr\"uchen.
\newblock Chelsea Publishing Co., New York, N. Y. (1950).
\newblock 2d ed

\bibitem{Picard}
Picard, E.: Sur une application de la th\'eorie des fonctions elliptiques.
\newblock C. R. A. S. \textbf{89}, 74--76 (1879)

\bibitem{Poincare}
Poincar{\'e}, H.: Sur les {E}quations {L}ineaires aux {D}ifferentielles
  {O}rdinaires et aux {D}ifferences {F}inies.
\newblock Amer. J. Math. \textbf{7}(3), 203--258 (1885)

\bibitem{Poole}
Poole, E.G.C.: Introduction to the theory of linear differential equations.
\newblock Dover Publications, Inc., New York (1960)

\bibitem{Sparre}
de~Sparre, C.: Sur l'\'equation...
\newblock Acta Mathematica \textbf{3}, 105--140, 289--321 (1883)

\bibitem{Take}
Takemura, K.: The {H}eun equation and the {C}alogero-{M}oser-{S}utherland
  system. {I}. {T}he {B}ethe {A}nsatz method.
\newblock Comm. Math. Phys. \textbf{235}(3), 467--494 (2003)

\bibitem{Take2}
Takemura, K.: On the {H}eun equation.
\newblock Philos. Trans. R. Soc. Lond. Ser. A Math. Phys. Eng. Sci.
  \textbf{366}(1867), 1179--1201 (2008)

\bibitem{Take3}
Takemura, K.: The {H}ermite-{K}richever ansatz for {F}uchsian equations with
  applications to the sixth {P}ainlev\'e equation and to finite-gap potentials.
\newblock Math. Z. \textbf{263}(1), 149--194 (2009)

\bibitem{TV}
Treibich, A., Verdier, J.L.: Solitons elliptiques.
\newblock In: The {G}rothendieck {F}estschrift, {V}ol.\ {III}, pp. 437--480.
  Birkh\"auser Boston, Boston, MA (1990)

\bibitem{Verdier}
Verdier, J.L.: New elliptic solitons.
\newblock In: Algebraic analysis, {V}ol.\ {II}, pp. 901--910. Academic Press,
  Boston, MA (1988)

\bibitem{Veselov}
Veselov, A.P.: On {D}arboux-{T}reibich-{V}erdier potentials.
\newblock Lett. Math. Phys. \textbf{96}(1-3), 209--216 (2011)

\bibitem{WW}
Whittaker, E.T., Watson, G.N.: A {C}ourse of {M}odern {A}nalysis.
\newblock Cambridge University Press, Cambridge (1996)

\end{thebibliography}

\end{document}